\documentclass[12pt,twoside,a4paper]{article}
\usepackage[utf8]{inputenc}
\usepackage{amsmath, amssymb, mathtools, enumerate} 
\usepackage{amsfonts, amsthm, stmaryrd, fullpage}
\usepackage{authblk}
\usepackage{tikz-cd}
\usepackage[german, english]{babel}

\newtheorem{thm}{Theorem}[section]
\newtheorem{lemma}[thm]{Lemma}
\newtheorem{cor}[thm]{Corollary}
\newtheorem{prop}[thm]{Proposition}

\theoremstyle{definition} 
\newtheorem{mydef}[thm]{Definition}  
\newtheorem{example}[thm]{Example}
\newtheorem{question}[thm]{Question} 

\theoremstyle{remark}
\newtheorem{rmk}[thm]{Remark}

\newcommand\Ainf{{\rm inf}}

\newcommand\Core{{\rm Core}}
\newcommand\sCore{{\rm Core_{s}}}

\newcommand\im{{\rm im}}

\newcommand\perfd{{\rm perfd}}

\newcommand\spc{{\rm sp}}

\newcommand\Spa{{\rm Spa}}
\newcommand\Spec{{\rm Spec}}
\newcommand\supp{{\rm supp}} 
\newcommand\TopSpec{{\rm Spec_{Top}}}
\newcommand\Zar{{\rm Zar}}

\begin{document}

\title{Topological spectrum and perfectoid Tate rings}
\author{Dimitri Dine}
\date{}
\maketitle

\begin{abstract} 
We study the topological spectrum $\TopSpec(R)$ of a seminormed ring $R$ which we define as the space of prime ideals $\mathfrak{p}$ such that $\mathfrak{p}$ equals the kernel of some bounded power-multiplicative seminorm. For any seminormed ring $R$ we show that $\TopSpec(R)$ is a quasi-compact sober topological space. When $R$ is a perfectoid Tate ring we construct a natural homeomorphism \begin{equation*} \TopSpec(R) \simeq \TopSpec(R^{\flat}) \end{equation*} between the topological spectrum of $R$ and the topological spectrum of its tilt $R^{\flat}$. As an application, we prove that a perfectoid Tate ring $R$ is an integral domain if and only if its tilt is an integral domain.           
\end{abstract}

\section{Introduction}

In his thesis \cite{Scholze} introducing the revolutionary concept of a perfectoid space, Scholze constructed a functor $A \mapsto A^{\flat}$ from the category of certain Banach rings, called perfectoid rings, of characteristic $0$ to the category of perfect Banach rings of characteristic $p$. He proved that for any given perfectoid field $K$ this so-called tilting functor defines an equivalence between the category of perfectoid Banach $K$-algebras and the category of perfectoid (equivalently, perfect) Banach algebras over $K^{\flat}$. The tilt $A^{\flat}$ of a perfectoid ring $A$ can be described explicitly as the set of sequences $(f^{(n)})_{n \geq 0}$ of elements of $A$ such that $f^{(n+1)}$ is a $p$-th root of $f^{(n)}$ for each $n$, endowed with the ring operations $fg = (f^{(n)}g^{(n)})_{n \geq 0}$ and $f+g=((f+g)^{(n)})_{n \geq 0}$ where $(f+g)^{(n)}$ is defined as $(f+g)^{(n)}=\lim_{m \to \infty}(f^{(n+m)} + g^{(n+m)})^{p^m}$. There results a natural map $f \mapsto f^{\#} := f^{(0)}$ from $A^{\flat}$ to $A$ (this map is multiplicative, but not additive) inducing a map $x \mapsto x^{\flat}$ between the adic spectrum $\Spa(A, A^{+})$ of a perfectoid affinoid $K$-algebra $(A,A^{+})$ and the adic spectrum of its tilt $(A^{\flat},A^{\flat+})$. Namely, we can define $x^{\flat}$ by $f \mapsto |f(x^{\flat})| := |f^{\#}(x)|$. In Corollary 6.7 of loc.~cit.~Scholze proved that the map $x \mapsto x^{\flat}$ is in fact a homeomorphism identifying rational subsets. The analogous map from the Berkovich spectrum $\mathcal{M}(A)$ of $A$ to that of $A^{\flat}$ is a homeomorphism as well. Moreover, as observed by Kedlaya in \cite{Kedlaya17}, Corollary 2.4.5, the tilting equivalence induces a bijection between the set of closed ideals $I$ of $A$ with $A/I$ uniform and the set of closed perfect ideals of $A^{\flat}$. 

However, it is a priori not easy to determine whether the quotient $A/I$ is a uniform Banach algebra for some given closed ideal $I$ in $A$. For example, it is a priori a non-trivial question whether for every maximal ideal $\mathfrak{m}$ in a perfectoid algebra $A$ the field $A/\mathfrak{m}$ must be a uniform Banach field (if it is, then it is automatically a perfectoid field by Theorem 3.7 or Theorem 4.2 in \cite{Kedlaya18}). Furthermore, the map $f \mapsto f^{\#}$ is far from being surjective, so it is in general rather hard to relate the ideal structure or algebraic properties of a perfectoid algebra $A$ to those of $A^{\flat}$. E.g., if $A^{\flat}$ is an integral domain, is $A$ necessarily an integral domain? In a similar vein, if $x \in \Spa(A,A^{+})$ is a continuous valuation such that $x^{\flat} \in \Spa(A^{\flat},A^{\flat +})$ is an absolute value, it is a priori unclear if $x$ must be itself an absolute value.

This paper is devoted to the study of ideals in perfectoid algebras and, more generally, in perfectoid Tate rings in the sense of Fontaine (see the definition in \cite{Fontaine}). In the following $\mathcal{M}(R)$ denotes the Berkovich spectrum of a seminormed ring $R$ and $\phi \mapsto \phi^{\#}$ is the inverse of the homeomorphism $\mathcal{M}(R) \simeq \mathcal{M}(R^{\flat}), \phi \mapsto \phi^{\flat}$, between the Berkovich spectrum of a perfectoid Tate ring $R$ and the Berkovich spectrum of its tilt.  We note that Berkovich \cite{Berkovich} only defined his spectrum (called the Gelfand spectrum in that work) for Banach rings but the same definition works also for general seminormed rings leading to the same basic properties. We call an ideal $I$ of a seminormed ring $R$ \textit{spectrally reduced} if $I$ is the kernel of some bounded power-multiplicative seminorm on $R$ (equivalently, $I$ is an intersection of prime ideals of the form $\ker(\phi)$ for some $\phi \in \mathcal{M}(R)$). To formulate our results we need a notion of 'spectrum' for a seminormed ring $R$ which takes into account the seminorm on $R$ but is still algebraic in nature.
\begin{mydef}[Definition \ref{Definition of the topological spectrum}]~
The topological spectrum $\TopSpec(R)$ of a seminormed ring $R$ is the set of spectrally reduced prime ideals of $R$, endowed with its subspace topology induced from the Zariski topology on $\Spec(R)$. 
\end{mydef} 
It follows from a theorem of Berkovich (\cite{Berkovich}, Theorem 1.2.1) that the topological spectrum is non-empty for any seminormed ring $R$.  When $R$ is a Banach ring (i.e., complete), every maximal ideal is spectrally reduced and thus the topological spectrum contains the set of all maximal ideals of $R$ (more general seminormed rings which enjoy the same property are characterized in a number of ways in Section 3). With regard to the topological properties of $\TopSpec(R)$ we prove the following result. 
\begin{thm}[Theorem \ref{Radical ideals in Noetherian Banach algebras}, Corollary \ref{Sobriety}, Theorem \ref{The topological spectrum is quasi-compact}] \label{Zeroth theorem} For any seminormed ring $R$ the topological space $\TopSpec(R)$ is quasi-compact and sober. If $R$ is a Noetherian Banach algebra over some nonarchimedean field, then every radical ideal of $R$ is spectrally reduced and thus $\TopSpec(R)$ coincides with $\Spec(R)$. 
\end{thm} 
From the second statement of the theorem we deduce the following structural result for Noetherian Banach algebras. It is classically known for affinoid algebras, both in the sense of Tate and in the broader sense of Berkovich, but appears to be new in general. 
\begin{cor}[Corollary \ref{Noetherian Banach algebras are Jacobson}] \label{Zeroth corollary} Every Noetherian Banach algebra over a nonarchimedean field is a Jacobson ring. 
\end{cor}
In non-Noetherian situations the topological spectrum can differ quite dramatically from the usual prime ideal spectrum. For example, if $K$ is a nonarchimedean field with dense valuation and $X$ is an infinite totally disconnected compact Hausdorff space, then every closed prime ideal in the Banach algebra $C(X, K)$ of continuous $K$-valued functions on $X$ is already maximal (\cite{Kaplansky}, Theorem 29) but there exist non-maximal prime ideals in $C(X, K)$ (loc.~cit., example after the proof of Theorem 30).    

Our main result relates the topological spectrum of a perfectoid Tate ring $R$ to the topological spectrum of its tilt. It leads to affirmative answers to all of the questions raised at the beginning of this introduction.
\begin{thm}[Theorem \ref{Spectrally reduced quotients are perfectoid}, Proposition \ref{Spectrally reduced ideals and tilting}, Theorem \ref{Prime ideals and tilting}] \label{First theorem} Let $R$ be a perfectoid Tate ring with tilt $R^{\flat}$ and let $I$ be a spectrally reduced ideal of $R$. Then the quotient ring $R/I$, equipped with the quotient norm, is a perfectoid Tate ring. There is an inclusion-preserving bijection $I \mapsto I^{\flat}$ between the set of spectrally reduced ideals of $R$ and the set of spectrally reduced ideals of $R^{\flat}$ given by the maps \begin{equation*} I \mapsto I^{\flat} := \{\, f = (f^{(n)})_{n \in \mathbb{N}_{0}} \mid f^{(n)} \in I \, \text{for all}~n\} \end{equation*} and \begin{equation*} J \mapsto J^{\#} := \bigcap_{\substack{\phi \in \mathcal{M}(R^{\flat}) \\ J \subseteq \ker(\phi)}}\ker(\phi^{\#}). \end{equation*} This bijection sends prime ideals to prime ideals and thus defines a homeomorphism \begin{equation*} \TopSpec(R) \simeq \TopSpec(R^{\flat}). \end{equation*}\end{thm}
Note that if $I$ is any ideal of $R$, then the analog of the set $I^{\flat}$ which we used in the above theorem for spectrally reduced ideals need not a priori be an ideal of $R^{\flat}$, due to the highly transcendental nature of the definition of the sum $f + g$ of two elements $f, g \in R^{\flat}$. However, it is an ideal when $I$ is assumed to be a \textit{closed} ideal of $R$, and spectrally reduced ideals are always closed. The first assertion of Theorem \ref{First theorem} (concerning quotients by spectrally reduced ideals) is essentially a formal consequence of a recent result due to Bhatt and Scholze (\cite{Prismatic}, Theorem 7.4) whose proof relies on a key flatness lemma first proved by Andr\'{e} in the course of his work on the direct summand conjecture (see \cite{Andre18-2}, \S 2.5, or \cite{Bhatt}, Theorem 2.3). Assuming this first assertion, the bijection in our theorem coincides with the one introduced in \cite{Kedlaya17}, Corollary 2.4.5, since it is known that closed radical ideals in a perfectoid Tate ring of characteristic $p$ are spectrally reduced (in loc.~cit., Kedlaya works with closed perfect ideals of $R^{\flat}$; but if $J \subsetneq R^{\flat}$ is a perfect ideal and $f \in R^{\flat}$ is an element with $f^{m} \in J$, then there exists an integer $k \geq 0$ such that $p^{k} \geq m$ and then $f^{p^{k}} = f^{p^{k}-m}f^{m} \in J$ showing that perfect ideals in $R^{\flat}$ are the same as radical ideals). However, the fact that this bijection identifies the sets of spectrally reduced \textit{prime} ideals was not observed in \cite{Kedlaya17} and appears to be new. As a consequence of Theorem \ref{First theorem} we obtain the following result. 
\begin{cor}[Corollary \ref{Integral domains and tilting}] \label{First corollary} A perfectoid Tate ring $R$ is an integral domain if and only if its tilt $R^{\flat}$ is an integral domain.\end{cor}
We also derive several other consequences from Theorem \ref{First theorem}. For example, we show that a bounded multiplicative seminorm $\phi \in \mathcal{M}(R)$ (or a continuous valuation $x \in \Spa(R, R^{+})$, where $R^{+}$ is a ring of integral elements) is an absolute value if and only if the corresponding seminorm or valuation on $R^{\flat}$ is an absolute value (Corollary \ref{Absolute values and tilting}, Corollary \ref{Absolute values of higher rank and tilting}). By comparing Shilov boundaries of a perfectoid Tate ring and its tilt we deduce from this statement that a perfectoid Tate ring $R$ which contains a nonarchimedean field has no non-zero topological divisors of zero if and only if the same property holds true for its tilt $R^{\flat}$ (Corollary \ref{Topological zero-divisors and tilting}). \newline    
\linebreak
\textbf{Outline.} Let us conclude this introduction with a brief overview of the individual sections. In Section 2 we explain first basic properties of the topological spectrum of a seminormed ring. We introduce the notion of the core of a normed ring and use it to prove the second part of Theorem \ref{Zeroth theorem}~and Corollary \ref{Zeroth corollary}. Section 3 continues our discussion of seminormed rings and algebras begun in Section 2. It is devoted to Zariskian seminormed rings, a class of seminormed rings which is strictly bigger than the class of Banach rings but whose members still share some of the pleasant properties of the latter. We use a notion of 'Zariskisation' to study properties of the space $\TopSpec(R)$ for a general seminormed ring $R$. In particular, we prove the first part of Theorem \ref{Zeroth theorem}~saying that the topological spectrum is quasi-compact and sober. In Section 4 we prove our results on perfectoid Tate rings; in particular, we prove Theorem \ref{First theorem}~and Corollary \ref{First corollary}. \newline        
\linebreak
\textbf{Notation and terminology.} All rings considered in this paper will be commutative and with $1$. We will regard the element $0$ in a ring as a zero-divisor. Recall that a seminorm on an abelian group $G$ is a function $\lVert \cdot \rVert: G \to \mathbb{R}_{\geq 0}$ such that $\lVert 0 \rVert = 0$ and such that the triangle inequality $\lVert f + g \rVert \leq \lVert f \rVert + \lVert g \rVert$ holds for any two elements $f, g \in G$. A seminorm is said to be nonarchimedean if for any $f, g \in G$ the stronger inequality $\lVert f + g \rVert \leq \max(\lVert f \rVert, \lVert g \rVert)$ is satisfied. Unless explicitly stated otherwise we only consider nonarchimedean seminorms in this paper. A seminorm $\lVert \cdot \rVert$ is called a norm if $0$ is the only element of $G$ which is mapped to $0$ under $\lVert \cdot \rVert$. Two seminorms $\lVert \cdot \rVert_{1}$ and $\lVert \cdot \rVert_{2}$ are called equivalent if there exist constants $C_{1}, C_{2} > 0$ such that \begin{equation*} \lVert f \rVert_{2} C_{2} \leq \lVert f \rVert_{1} \leq C_{1} \lVert f \rVert_{2} \end{equation*}holds for all elements $f \in G$. 

For a ring $A$, a seminorm on $A$ is a seminorm on the underlying additive abelian group of $A$ such that $\lVert 1 \rVert = 1$. A seminorm $\lVert \cdot \rVert$ on a ring $A$ is said to be submultiplicative (or a ring seminorm) if it is compatible with multiplication in the sense that \begin{equation*} \lVert fg \rVert \leq \lVert f \rVert \lVert g \rVert \end{equation*}for any two elements $f, g \in A$. For any abelian group $G$ a seminorm $\lVert \cdot \rVert$ on $G$ equips $G$ with the structure of a topological group by letting the open balls around $0$ with respect to this seminorm form a fundamental system of open neighbourhoods of zero. Thereby two equivalent seminorms always define the same topology on $G$. If $\lVert \cdot \rVert$ is a submultiplicative seminorm on a ring $A$, then the resulting topological group on $A$ is even a topological ring (that is, the multiplication map $A \times A \to A$ is continuous). By a (semi)normed ring (respectively, a Banach ring) $(A, \lVert \cdot \rVert)$ we will mean a ring $A$ equipped with a submultiplicative (semi)norm (resp., a complete submultiplicative seminorm) $\lVert \cdot \rVert$. In this paper we tacitly assume all our seminorms to be submultiplicative. We often also tacitly identify equivalent seminorms on a ring $A$. Oftentimes, we also suppress the seminorm $\lVert \cdot \rVert$ from the notation for a seminormed ring and simply write $A$ instead of $(A, \lVert \cdot \rVert)$. In such cases, a choice of seminorm (or an equivalence class of seminorms) on $A$ is to be understood.    

The trivial norm on a ring $A$ is the norm defined by $\lVert f \rVert = 1$ for all $f \in A, f \neq 0$. In this paper we only consider nontrivially normed rings. A seminorm on a ring $A$ is said to be power-multiplicative (resp., multiplicative) if $\lVert f^{n} \rVert = \lVert f \rVert^{n}$ for all $f$ and all integers $n \geq 0$ (resp., $\lVert fg \rVert = \lVert f \rVert \lVert g \rVert$ for all $f, g \in A$). A multiplicative norm on a ring $A$ is also called an absolute value. By a nonarchimedean field we mean a field $K$ complete with respect to a (nonarchimedean) absolute value on it. 

A homomorphism $\varphi: G \to H$ between seminormed abelian groups is said to be bounded if there exists a constant $C > 0$ with $\lVert \varphi(g) \rVert \leq C \lVert g \rVert$ for all $g \in G$. It is said to be submetric if the above constant $C$ can be taken to be equal to $1$. Every bounded homomorphism is continuous and the converse is also true for homomorphisms between seminormed spaces over a nonarchimedean field $K$. Indeed, the latter assertion is proved for normed $K$-vector spaces in \cite{BGR}, Corollary 2.1.8/3, and the general case follows from this by observing that every continuous $K$-linear map $f: V \to W$ of seminormed $K$-vector spaces sends elements of seminorm $0$ to elements of seminorm $0$ and then applying loc.~cit.~to the induced continuous map of normed $K$-vector spaces \begin{equation*} \overline{f}: V/\ker(\lVert\cdot\rVert_{V}) \to W/\ker(\lVert\cdot\rVert_{W}). \end{equation*}For a seminormed ring $(A, \lVert \cdot \rVert)$ we call another seminorm $\phi$ on $A$ bounded if the identity homomorphism is bounded as a homomorphism from $(A, \lVert \cdot \rVert)$ to $(A, \phi)$. In other words, $\phi$ is called bounded if there exists a constant $C > 0$ such that $\phi(f) \leq C \lVert f \rVert$ for all $f \in A$. If $\varphi: A \to B$ is a bounded homomorphism between seminormed rings and the seminorm on $A$ is power-multiplicative, then $\varphi$ is automatically submetric (see, for example, \cite{BGR}, Proposition 1.3.1/2). In particular, a power-multiplicative bounded seminorm $\phi$ on a seminormed ring $(A, \lVert \cdot \rVert)$ satisfies $\phi(f) \leq \lVert f \rVert$ for all $f \in A$. 

A seminormed ring $(A, \lVert \cdot \rVert)$ is called uniform if its topology can be defined by a bounded power-multiplicative seminorm. For every seminormed ring $A$ there exists a maximal bounded power-multiplicative seminorm on $A$, called the spectral seminorm. We refer the reader to \cite{Berkovich}, 1.3, or to (the discussion preceding) Remark \ref{The spectral seminorm for non-complete rings} for its definition and alternative descriptions. A seminormed ring is uniform if and only if the spectral seminorm defines the same topology on it as the original seminorm of $A$. For a seminormed ring $(A, \lVert \cdot \rVert)$ we denote by $A^{u}$ the completion of $A$ with respect to its spectral seminorm. The uniform Banach ring $A^{u}$ is called the uniformization of $A$.   

A Huber ring $A$ is a topological ring that contains an open adic subring $A_0$ (the so-called ring of definition) with a finitely generated ideal of definition. A Huber ring $A$ is called a Tate ring if it contains a topologically nilpotent unit. We sometimes refer to a topologically nilpotent unit in a Tate ring as a pseudo-uniformizer. In the sequel we will use without further comment the following well-known relationship between Tate rings and seminormed rings (see, for example, \cite{Kedlaya-Liu}, Remark 2.4.4). If $A$ is a Tate ring (respectively, a Hausdorff Tate ring) with topologically nilpotent unit $\varpi$, then, for any $c \in (0,1)$, 
\begin{equation*} \lVert x \rVert := \inf \{\, c^n : n \in \mathbb{Z}, x \in \varpi^{n}A_{0} \,\},  \\\ x \in A \end{equation*} 
defines a seminorm (respectively, a norm) which induces the topology of $A$ (more generally, one can equip any Huber ring with a natural structure of a seminormed ring, as is explained for Banach rings in \cite{Kedlaya17}, Remark~1.5.4). If our Tate ring $A$, with ring of definition $A_{0}$, is a $K$-algebra over some nonarchimedean field $K$ and the inclusion $K \hookrightarrow A$ is continuous, then $A$ carries a natural structure of a seminormed $K$-algebra. Indeed, if for every $\gamma \in |K^{\times}|$ we choose $t_{\gamma} \in K$ with $|t_{\gamma}| = \gamma$, then \begin{equation*} \lVert x \rVert = \inf \{\, \gamma \mid x \in t_{\gamma}A_{0} \,\}, \\\ x \in A, \end{equation*} is a $K$-algebra seminorm which induces the topology of $A$. Conversely, every seminormed ring with a topologically nilpotent unit is a Tate ring, where \begin{equation*} A_{\leq 1} = \{\, x \in A \mid \lVert x \rVert \leq 1 \,\} \end{equation*} can be taken as ring of definition. For example, if $A$ is a normed $K$-algebra over some nonarchimedean field $K$, then $A$ is a Tate ring, where we can take any $\varpi \in K^{\times}$ with $|\varpi| < 1$ as topologically nilpotent unit. We call an element $f$ in a topological ring $A$ power-bounded if the set of all powers of $f$ is a bounded subset of $A$ and we call a Huber ring $A$ uniform if its subring $A^{\circ}$ consisting of power-bounded elements is bounded. If $A$ is a seminormed ring whose underlying topological ring is a Huber ring, then this Huber ring is uniform in the above sense if and only if $A$ is uniform as a seminormed ring in the sense of the preceding paragraph (see, for example, \cite{Kedlaya17}, Exercise~1.5.13). We denote by $\widehat{A}$ the completion of a seminormed ring or Huber ring $A$. In this paper complete topological rings will always be Hausdorff and the term 'completion' will always refer to the Hausdorff completion of a topological ring. \newline 
\linebreak 
\textbf{Acknowledgements.} The author expresses his gratitude to Eva Viehmann for her reading of this manuscript as well as for her encouragement and many helpful discussions on the topics of this work. The author is also grateful to Peter Scholze for important comments on a draft version of this paper and, in particular, for his comments with regard to the main results on perfectoid rings. Finally, the author would like to thank Tomoki Mihara for suggesting an improvement to the proof of Theorem 2.23 and the anonymous referee for careful reading of the paper and useful remarks.

\section{The core of a nonarchimedean Banach algebra}

Recall that a map $\varphi: G \to H$ of seminormed abelian groups is said to be strict if the induced isomorphism 
\begin{equation*} G/\ker(\varphi) \to \im(\varphi) \end{equation*}is a homeomorphism or, equivalently, if the quotient seminorm defines the same topology on $\im(\varphi)$ as the seminorm induced from $H$. Clearly, a continuous map $\varphi: G \to H$ is strict if and only if it is open onto its image.

In his book \cite{Warner} on topological fields Warner defines the core of a norm $N$ on a ring $A$ as the set of elements $x \in A$ such that $N(xy)=N(x)N(y)$ for all $y \in A$. The next definition can be seen as a generalisation of this concept.
\begin{mydef}[The core of a normed ring] \label{Def. of the core} Given a normed ring $(A,\lVert \cdot \rVert)$ we define the strict core of $(A,\lVert \cdot \rVert)$ as the subset \begin{equation*} \sCore(A,\lVert \cdot \rVert) := \{\, x \in A \mid \text{the map  $A \to A, a \mapsto ax$, is strict}\, \}. \end{equation*} 
We then define the simple core (also to be called the core) of $A$ as \begin{equation*} 
\Core(A,\lVert \cdot \rVert) := \sCore(A,\lVert \cdot \rVert) \cup \{\, x \in A \mid x \ \text{is a zero-divisor in} \ A \, \}. 
\end{equation*}\end{mydef}
Note that the above definition depends only on the norm topology of a given normed ring and not on a particular choice of norm. We will write simply $\sCore(A)$ and $\Core(A)$ when there is no possibility of confusion on the norm topology. 

If $V$, $W$ are Banach spaces over a nonarchimedean field $K$ and $\phi : V \to W$ is a bounded $K$-linear map, then Banach's Open Mapping Theorem tells us that $\phi$ is strict if and only if it has closed image. By the more general version of the Open Mapping Theorem proved in \cite{Henkel} the same holds true for any $A$-linear continuous map between complete metrizable topological $A$-modules over a complete Tate ring $A$. Hence, for any complete Tate ring $A$ and any $x \in A$, $x \in \sCore(A)$ if and only if the principal ideal $(x)_A$ generated by $x$ is closed.
\begin{example} \label{Normed rings with full core}~\begin{enumerate} 
\item If $A$ is a Noetherian complete Tate ring, then every ideal of $A$ is closed and hence $\sCore(A) = A$. 
\item If $A$ is a normed ring whose underlying ring is an integral domain, then for any $x \in A$ the quotient norm on $(x)_{A} = \im(f \mapsto fx)$ is given by \begin{equation*} \lVert fx \rVert = \lVert f \rVert \end{equation*}for any $f \in A$. Hence if the norm on $A$ is multiplicative, then the quotient norm on $(x)_{A}$ is equivalent to the restriction of the norm $\lVert \cdot \rVert$ on $A$, so we have the equality $\sCore(A) = A$ for any multiplicatively normed ring. 
\item Let $F$ be a perfectoid field of characteristic $p>0$ with valuation ring $\mathcal{O}_{F}$, and let $\varpi_{F} \in F^{\times}$ be a non-zero element of norm $<1$. Consider $\mathbb{A}_{\Ainf} = W(\mathcal{O}_{F})$, the ring of $p$-typical Witt vectors of $\mathcal{O}_{F}$, endowed with the $([\varpi_{F}], p)$-adic norm. By \cite{Fargues-Fontaine}, Propositions 1.4.9 and 1.4.11, the $([\varpi_{F}], p)$-adic topology on $\mathbb{A}_{\Ainf}$ can also be defined by a multiplicative norm, so we have $\sCore(\mathbb{A}_{\Ainf}) = \mathbb{A}_{\Ainf}$.\end{enumerate}
\end{example} 
The notion of the core is intimately related to the notion of a topological divisor of zero which was first introduced by Shilov (\cite{Gelfand-Raikov-Shilov}) and which comes up in the study of Banach algebras both over a nonarchimedean field $K$ and over $\mathbb{C}$.
\begin{mydef} (Topological divisor of zero) Let $A$ be a normed ring. An element $x \in A$ is called a topological divisor of zero if there exists a sequence $(x_{\lambda})_{\lambda}$ in $A$ such that $\inf_{\lambda} \lVert x_{\lambda} \rVert>0$ and $\lim_{\lambda}x_{\lambda}x = 0$. 
\end{mydef} \begin{lemma} \label{The core and top. divisors of zero} Let $A$ be a normed ring and let $x \in A$ such that $x$ is not a divisor of zero. Then $x \in \Core(A)$ if and only if $x$ is not a topological divisor of zero.\end{lemma}
\begin{proof} If $x \in A$ is not a zero-divisor, then the map $f \mapsto fx$ defines a group isomorphism $A \to Ax$. That is, there exists an inverse homomorphism which recovers $f$ from $fx$. Thus the map $f \mapsto fx$ is strict if and only if this inverse homomorphism is continuous. But a homomorphism of topological groups is continuous if and only if it is continuous at $0$. Hence $f \mapsto fx$ is strict if and only if for every sequence $(f_{\lambda})_{\lambda}$ such that $(f_{\lambda}x)_{\lambda}$ converges to $0$ we have that $(f_{\lambda})_{\lambda}$ must converge to $0$ itself.\end{proof} The following examples are taken from the lecture notes \cite{Kedlaya17} by Kedlaya.
\begin{example} 
(loc. cit., Lemma 1.5.26) Consider a uniform Banach ring $A$ and let $x = \sum_{n=0}^{\infty}x_{n}T^{n} \in A\langle T \rangle$ such that the $x_{n}$ generate the unit ideal in $A$. Then $x \in \sCore(A\langle T \rangle)$. In particular, if $A$ is a uniform Banach field, then $\sCore(A\langle T \rangle) = A\langle T \rangle$.  
\end{example} 
\begin{example}  
(loc. cit., Exercise 1.1.17) Let $p$ be a prime number. Let $A$ be the quotient of the infinite Tate algebra ${\mathbb{Q}_p}\langle T, U_1, V_1, U_2, V_2, \dots \rangle$ by the closure of the ideal $(TU_1 - pV_1, TU_2 - p^{2}V_2, \dots )$. Then $T$ is not a zero-divisor in $A$, but the ideal $(T)_A$ is not closed in $A$. Hence $T \in A$ is a topological divisor of zero which is not a divisor of zero; in particular, $\Core(A)  \neq A$. \end{example}In the rest of this section we study some properties of normed rings (and, in particular, of Banach algebras) satisfying $\Core(A) = A$ and use these properties to elucidate the structure of the topological spectrum (to be introduced in Definition \ref{Definition of the topological spectrum}) of certain Banach algebras $A$. \begin{lemma} \label{The core and norm inequalities} Let $K$ be a nonarchimedean field. A non-zero-divisor $x$ in a normed $K$-algebra $A$ belongs to $\Core(A)$ if and only if there exists a real number $c > 0$ such that 
\begin{equation*} \lVert f \rVert \leq c \lVert fx \rVert  \\\ \hspace{2cm} \text{for all} \ f \in A. \end{equation*} \end{lemma} 
\begin{proof} We have seen that, for $x$ a non-zerodivisor in a normed algebra $A$ over $K$ with $x \in \Core(A)$, the inverse map to $A \to (x)_A, f \mapsto fx$, is continuous. But $(x)_A$ is a normed vector space over $K$ and a $K$-linear map between normed $K$-vector spaces is continuous if and only if it is bounded. This proves the lemma. \end{proof} In the following we will set, for every $x \in \Core(A)$ not a zero-divisor, \begin{equation*} c_x := \inf \{\, c>0 \mid \lVert f \rVert \leq c \lVert fx \rVert \\\ \forall{f \in A}\, \} \end{equation*} and we will set $C_x := \lVert x \rVert c_x$.
One of the notions of spectrum in the setting of normed (and seminormed) rings $A$ is the Berkovich spectrum, whose definition we recall below in slightly greater generality than is common in the literature.\begin{mydef}[Berkovich spectrum] For any seminormed ring $A$ the Berkovich spectrum $\mathcal{M}(A)$ of $A$ is the space of bounded multiplicative seminorms on $A$ endowed with the weakest topology making the evaluation maps $\phi \mapsto \phi(f)$ continuous for every $f \in A$. Equivalently, $\mathcal{M}(A)$ is endowed with the subspace topology induced from the product topology on $\mathbb{R}_{\geq 0}^{A}$. \end{mydef}For Banach rings the Berkovich spectrum was studied by Berkovich in \cite{Berkovich} under the name Gelfand spectrum (in fact, in the case of normed algebras over a nonarchimedean field this notion had also been considered before that by Guennebaud in \cite{Guennebaud}). By Thm.~1.2.1 of \cite{Berkovich} the topological space $\mathcal{M}(A)$ is a non-empty compact Hausdorff space whenever $A$ is a Banach ring. For the sake of completeness we give (in the proof of Lemma \ref{Berkovich spectrum}) a brief argument reducing the analogous assertions for general seminormed rings to the case of Banach rings. To this end, we first record the following simple lemma. 
\begin{lemma} \label{Berkovich spectrum of a quotient} Let $J \subsetneq A$ be an ideal in a seminormed ring $A$ and suppose that $J$ is not dense in $A$. There is a natural bijection $\phi \mapsto \widetilde{\phi}$ between the set of bounded power-multiplicative seminorms $\phi$ on $A$ with $\ker(\phi) \supseteq J$ and the set of bounded power-multiplicative seminorms on $A/J$. This bijection sends multiplicative seminorms to multiplicative seminorms. Moreover, for any bounded power-multiplicative seminorm $\phi$ on $A$ with $\ker(\phi) \supseteq J$ we have $\ker(\widetilde{\phi}) = \ker(\phi)/J$. \end{lemma}
\begin{proof} Since $J$ is not dense in $A$, the quotient $A/J$ is a seminormed ring when endowed with its quotient seminorm (recall that a seminorm is required to satisfy $\lVert 1 \rVert \neq 0$, so the condition that $J$ is not dense is really necessary). Given a seminorm $\phi$ on $A$ with $\ker(\phi) \supseteq J$ we define a power-multiplicative seminorm $\widetilde{\phi}$ on $A/J$ by \begin{equation*} \widetilde{\phi}(f + J) = \phi(f) \end{equation*}for $f \in A$. It is clear that the seminorm $\widetilde{\phi}$ on $A/J$ is power-multiplicative (resp., multiplicative) if and only if $\phi$ is power-multiplicative (resp., multiplicative). We claim that $\phi$ is bounded if and only if $\widetilde{\phi}$ is bounded with respect to the quotient norm on $A/J$. In order to show this, assume that $\phi$ is bounded and let $f \in A$. If $g \in A$ such that $g - f \in J$, then $\widetilde{\phi}(f + J) = \widetilde{\phi}(g + J) \leq \lVert g \rVert$. It follows that \begin{equation*} \widetilde{\phi}(f + J) \leq \inf_{\substack{g \in A \\ g - f \in J}}\lVert g \rVert = \lVert f + J \rVert, \end{equation*}that is, $\widetilde{\phi}$ is bounded with respect to the quotient norm on $A/J$. Conversely, if $\widetilde{\phi}$ is bounded with respect to the quotient norm, then $\phi(f) = \widetilde{\phi}(f + J) \leq \lVert f + J \rVert \leq \lVert f \rVert$ for every $f \in A$, showing that $\phi$ is bounded. 

On the other hand, for a power-multiplicative seminorm $\alpha$ on $A/J$ we can define a power-multiplicative seminorm $\alpha'$ on $A$ with $\ker(\alpha') \supseteq J$ by putting \begin{equation*} \alpha'(f) = \alpha(f + J) \end{equation*}for $f \in A$. We see by inspection that the two maps $\phi \mapsto \widetilde{\phi}$ and $\alpha \mapsto \alpha'$ are inverse to each other. The last assertion of the lemma follows immediately from the construction of the map $\phi \mapsto \widetilde{\phi}$.\end{proof}
\begin{lemma} \label{Berkovich spectrum} For any seminormed ring $A$ with completion $\widehat{A}$ there is a natural homeomorphism $\mathcal{M}(A) \simeq \mathcal{M}(\widehat{A})$. In particular, the Berkovich spectrum $\mathcal{M}(A)$ is a non-empty compact Hausdorff space. \end{lemma} \begin{proof} We first note that $\mathcal{M}(A)$ (whenever it is non-empty) must always be a Hausdorff space since it is endowed with the subspace topology induced from the product space $\mathbb{R}_{\geq 0}^{A}$ (recall that products of Hausdorff spaces are Hausdorff). If $A$ is a normed ring, then we have a natural bijection $\mathcal{M}(A) \simeq \mathcal{M}(\widehat{A})$, where the map $\mathcal{M}(\widehat{A}) \to \mathcal{M}(A)$ is restriction and the map $\mathcal{M}(A) \to \mathcal{M}(\widehat{A})$ is extension of a bounded multiplicative seminorm $\phi$ by continuity from $A$ to its completion $\widehat{A}$. By \cite{Berkovich}, Thm.~1.2.1, the space $\mathcal{M}(\widehat{A})$ is non-empty and compact. Since every bijection from a compact Hausdorff space onto a Hausdorff space is a homeomorphism we see that $\mathcal{M}(A)$ is homeomorphic to $\mathcal{M}(\widehat{A})$. This proves the claim when $A$ is a normed ring. 

For an arbitrary seminormed ring $A$, the quotient $A/\overline{(0)}$, where $\overline{(0)}$ denotes the closure of the zero ideal, is a normed ring. Note that $\overline{(0)}$ is a proper ideal since $\lVert 1 \rVert = 1 \neq 0$. The kernel of any bounded multiplicative seminorm $\phi$ on $A$ is closed (since $\phi$ is continuous as a function $A \to \mathbb{R}_{\geq 0}$) and thus contains the ideal $\overline{(0)}$. It follows from Lemma \ref{Berkovich spectrum of a quotient} that we have a bijection $\mathcal{M}(A/\overline{(0)}) \simeq \mathcal{M}(A)$. By the first paragraph, $\mathcal{M}(A/\overline{(0)})$ is non-empty, compact and Hausdorff and $\mathcal{M}(A)$ is Hausdorff. Hence the above bijection must be a homeomorphism.\end{proof}We call an element $\varpi \in A$ in a seminormed ring $(A, \lVert\cdot\rVert)$ multiplicative if $\lVert\varpi f\rVert = \lVert\varpi\rVert\lVert f\rVert$ for all $f \in A$. The following lemma relates the Berkovich spectrum of a seminormed ring $A$ with a multiplicative topologically nilpotent unit to the set of all \textit{continuous} multiplicative seminorms on $A$.
\begin{lemma} \label{Continuity and boundedness} Let $(A, \lVert \cdot\rVert)$ be a seminormed ring and let $\phi$ be a continuous seminorm on $A$. Suppose that there exists a topologically nilpotent unit $\varpi \in A$ which is multiplicative both for $\lVert\cdot\rVert$ and $\phi$. Then there exists a real number $s>0$ such that $\phi^{s}$ is bounded with respect to $\lVert\cdot\rVert$. In particular, for every continuous multiplicative seminorm $\phi$ on a seminormed ring $A$ there exists some $s>0$ such that $\phi^{s} \in \mathcal{M}(A)$. \end{lemma}
\begin{proof} We follow the proof of Lemma 2.1.7 in \cite{Johansson-Newton2} (which is a corrected version of \cite{Johansson-Newton1}). By the continuity of $\phi$ we can find $\epsilon \in (0,1)$ such that $\lVert f\rVert \leq \epsilon$ implies $\phi(f) \leq 1$ for any $f \in A$. Given a topologically nilpotent unit $\varpi$ as in the statement of the lemma, we choose an integer $m>0$ with $\lVert \varpi^{m}\rVert \leq \epsilon$. We claim that there exists some $s>0$ such that \begin{equation*}\phi(f)^{s} \leq \lVert\varpi\rVert^{-2m}\lVert f\rVert\end{equation*}for all $f \in A$. In order to prove this claim let us first consider elements $f \in A$ with $\lVert f\rVert = 0$. To show that $\phi(f) = 0$ for such $f$ we observe that the kernel of any seminorm is precisely the closure of $\{0\}$ with respect to the topology defined by that seminorm. Since $\phi$ is continuous with respect to $\lVert \cdot\rVert$ (i.e., the topology on $A$ defined by $\phi$ is equal to or is weaker than the topology defined by $\lVert\cdot\rVert$), this entails $\ker(\lVert \cdot\rVert) \subseteq \ker(\phi)$. This settles the claim in the case $\lVert f\rVert=0$. Now, consider an element $f \in A$ with $\lVert f\rVert \neq 0$. If for every $n \in \mathbb{Z}$ we had $\lVert \varpi^{mn}\rVert\lVert f\rVert \leq 1$, then we would obtain $\lVert f\rVert \leq \lVert \varpi^{-mn}\rVert$ for all $n \in \mathbb{Z}$ and thus $\lVert f\rVert=0$. Hence there exists a lowest integer $n \in \mathbb{Z}$ such that $\lVert \varpi^{mn}\rVert\lVert f\rVert \leq 1$. In particular, we have $\lVert\varpi^{m(n-1)}\rVert\lVert f\rVert>1$. By our choice of $\epsilon$ and $m$ we have the estimate \begin{equation*}\lVert \varpi^{m(n+1)}f\rVert = \lVert\varpi^{m(n+1)}\rVert\lVert f\rVert \leq \lVert\varpi^{m}\rVert \leq \epsilon \end{equation*}and this entails $\phi(\varpi^{m(n+1)}f) \leq 1$. Let $s>0$ be such that $\phi(\varpi)^{s} = \lVert\varpi\rVert$. With this choice of $s>0$ we obtain \begin{equation*}\phi(f)^{s} \leq \phi(\varpi)^{-m(n+1)s} = \lVert\varpi\rVert^{-m(n+1)} = \lVert\varpi\rVert^{-2m}\lVert\varpi^{-m(n-1)}\rVert<\lVert\varpi\rVert^{-2m}\lVert f\rVert. \end{equation*}This establishes the claim made above and proves the first assertion of the lemma. As to the additional assertion about multiplicative seminorms, it follows from the first assertion since in this case every topologically nilpotent unit is multiplicative for $\phi$.  \end{proof} \begin{cor} \label{Continuity and boundedness for $K$-algebras}Let $(A, \lVert\cdot\rVert)$ be a seminormed algebra over a nonarchimedean field. For every continuous multiplicative seminorm $\phi$ on $A$ there exists $s>0$ such that $\phi^{s}$ is a bounded multiplicative seminorm. \end{cor}\begin{proof} Every $\varpi \in K^{\times}$ with $|\varpi|<1$ is a multiplicative topologically nilpotent unit of $(A, \lVert\cdot\rVert)$, so we can apply the last lemma.\end{proof}
\begin{rmk} \label{Uniform units} Some remarks are in order with regard to the hypotheses in the statement of Lemma \ref{Continuity and boundedness}. A multiplicative topologically nilpotent unit in a seminormed ring $(A, \lVert\cdot\rVert)$ is the same as a uniform unit in the sense of Kedlaya and Liu (\cite{Kedlaya-Liu}, Remark 2.3.9), i.e.~a topologically nilpotent unit $\varpi \in A$ such that $\lVert\varpi^{-1}\rVert = \lVert\varpi\rVert^{-1}$. Indeed, if $\varpi$ is multiplicative , then $1 = \lVert1\rVert = \lVert\varpi \varpi^{-1}\rVert = \lVert\varpi\rVert \lVert\varpi^{-1}\rVert$. Conversely, if $\lVert\varpi^{-1}\rVert = \lVert\varpi\rVert^{-1}$, then for every $f \in A$ we have \begin{equation*}\lVert\varpi\rVert \lVert f\rVert = \lVert \varpi\rVert \lVert\varpi^{-1}\varpi f\rVert \leq \lVert\varpi\rVert \lVert\varpi^{-1}\rVert\lVert \varpi f\rVert = \lVert \varpi f\rVert \end{equation*}which shows that $\varpi$ is a multiplicative element of $A$ (the opposite inequality is automatic, by the submultiplicativity of $\lVert\cdot\rVert$). We note that the equivalence between the notions of multiplicative topologically nilpotent units and uniform units has already been observed  by Johansson and Newton (see p.~8 of \cite{Johansson-Newton1}). We also note that by \cite{Kedlaya-Liu}, Remark 2.8.18, on every uniform Banach ring $A$ with a topologically nilpotent unit $\varpi$ we can find a norm defining the topology of $A$ which turns the given topologically nilpotent unit into a multiplicative one. This has the consequence that in a uniform complete Tate ring the notion of  spectrally reduced ideals (to be introduced below) does not depend on a particular choice of norm or equivalence class of norms (see Remark \ref{Topological spectrum might depend on the seminorm}). \end{rmk}
In the following we denote by $|\cdot|_{\spc}$ the spectral seminorm on $A$, i.e. \begin{equation*} |f|_{\spc} := \lim_{n \to \infty}{\lVert f^{n}\rVert}^{1/n} \hspace{2cm} \text{(for $f \in A$).}\end{equation*} It is well-known that the limit always exists, defines a seminorm on $A$ and that $|f|_{\spc} = \inf_{n}{\lVert f^{n} \rVert}^{1/n} = \sup_{\phi \in \mathcal{M}(A)}\phi(f)$, see \cite{Bourbaki}, 1.2.3, and \cite{Berkovich}, Thm.~1.3.1 (note that the above supremum is actually a maximum since $\mathcal{M}(A)$ is compact).
\begin{rmk} \label{The spectral seminorm for non-complete rings} Let $(A, \lVert \cdot \rVert)$ be any seminormed ring. By using the homeomorphism $\mathcal{M}(A) \simeq \mathcal{M}(\widehat{A})$ and applying \cite{Berkovich}, Thm.~1.3.1, to the Banach ring $\widehat{A}$ we see that \begin{equation*} |f|_{\spc} = \sup_{\phi \in \mathcal{M}(A)}\phi(f) \end{equation*}for all $f \in A$. Thus every power-multiplicative seminorm on a ring is a supremum over multiplicative seminorms. \end{rmk} 
\begin{mydef}[Spectrally reduced normed ring] We call a normed ring $A$ spectrally reduced if the spectral seminorm $|\cdot|_{\spc}$ is a norm. \end{mydef} By Remark \ref{The spectral seminorm for non-complete rings} above, $A$ is spectrally reduced if and only if \begin{equation*} \bigcap_{\phi \in \mathcal{M}(A)}\ker(\phi) = 0. \end{equation*} In this paper we will be particularly interested in the following class of ideals of a seminormed ring $A$.
\begin{mydef}[Spectrally reduced ideal] An ideal $J \subsetneq A$ of a seminormed ring $A$ is called spectrally reduced if there exists a power-multiplicative bounded seminorm $\phi$ on $A$ such that $\ker(\phi) = J$. 
\end{mydef} We record the following elementary properties of spectrally reduced ideals.
\begin{lemma}\label{Properties of spectrally reduced ideals} Let $A$ be a seminormed ring. \begin{enumerate}[(i)] \item Every spectrally reduced ideal is a closed radical ideal. \item An ideal $J \subsetneq A$ is spectrally reduced if and only if the normed ring $A/J$ is spectrally reduced.\end{enumerate}\end{lemma}
\begin{proof} To prove the first part of the lemma, note that a bounded seminorm $\phi$ on a seminormed ring $A$ is in particular a continuous map $A \to \mathbb{R}_{\geq 0}$, so its kernel must be closed in $A$. If $\phi$ is power-multiplicative, then $f^{n} \in \ker(\phi)$ for $f \in A$ and $n > 0$ means that $\phi(f)^{n} = \phi(f^{n}) = 0$ and thus $\phi(f) = 0$. Thus $\ker(\phi)$ is indeed a closed radical ideal. As for the second part of the lemma, we first observe that by Remark \ref{The spectral seminorm for non-complete rings} the spectral seminorm of a seminormed ring $A$ is the maximal bounded power-multiplicative seminorm on $A$. Thus a normed ring $A$ is spectrally reduced if and only if there exists some bounded power-multiplicative norm on $A$. Hence part (ii) of the lemma follows from Lemma \ref{Berkovich spectrum of a quotient}. \end{proof} From the above lemma and Remark \ref{The spectral seminorm for non-complete rings} we see that an ideal $J \subsetneq A$ is spectrally reduced if and only if \begin{equation*} \bigcap_{\substack{\phi \in \mathcal{M}(A) \\ J \subseteq \ker(\phi)}}\ker(\phi) = J. \end{equation*} We now define a suitable notion of spectrum of a seminormed ring $A$ which is our main object of study in this paper.
\begin{mydef}[Topological spectrum] \label{Definition of the topological spectrum} The topological spectrum of a seminormed ring $A$ is the subspace \begin{equation*} \TopSpec(A) = \{\, \mathfrak{p} \in \Spec(A) \mid \mathfrak{p} \ \textrm{spectrally reduced} \,\} \end{equation*}of $\Spec(A)$ endowed with the subspace topology (induced from the Zariski topology on $\Spec(A)$). 
\end{mydef}
\begin{rmk} \label{Topological spectrum might depend on the seminorm}Note that the notions of spectrally reduced ideal and topological spectrum a priori depend on the choice of an equivalence class of seminorms and not only on the topology of a seminormed ring $A$. However, if we restrict our attention to seminormed rings $(A, \lVert \cdot\rVert)$ which admit a multiplicative topologically nilpotent unit (this includes seminormed algebras over a nonarchimedean field but also all uniform complete Tate rings, by Remark \ref{Uniform units}), then one can use Lemma \ref{Continuity and boundedness} to show that these notions depend only on the topology. More precisely, we claim that in such a seminormed ring an ideal $J$ is spectrally reduced if and only if it can be written in the form \begin{equation*} J = \bigcap_{\phi}\ker(\phi) \end{equation*}where $\phi$ ranges over all \textit{continuous} multiplicative seminorms $\phi$ with $\ker(\phi) \supseteq J$. To prove this, it suffices to fix a continuous multiplicative seminorm $\phi$ and verify that its kernel $\ker(\phi)$ is a spectrally reduced ideal. Note that $\ker(\phi) = \ker(\phi^{s})$ for any $s>0$. But by Lemma \ref{Continuity and boundedness} there exists a constant $s>0$ such that $\phi^{s} \in \mathcal{M}((A, \lVert\cdot\rVert))$. It follows that $\ker(\phi) = \ker(\phi^{s})$ is indeed a spectrally reduced ideal of $(A, \lVert\cdot\rVert)$, as claimed. \end{rmk}
We note that Definition \ref{Definition of the topological spectrum} is very algebraic in nature but still captures some information on the topological and metric properties of the seminormed ring in question. A simpler definition would be to let the topological spectrum be the subset of $\Spec(A)$ consisting of all \textit{closed} prime ideals of $A$. We do not know whether the resulting space would coincide with our notion of topological spectrum. More generally, we could ask whether the converse to Lemma \ref{Properties of spectrally reduced ideals}(1) holds true.   
\begin{question} \label{Question on spectrally reduced ideals} Is every closed radical ideal in a seminormed ring spectrally reduced? \end{question}As long as Question \ref{Question on spectrally reduced ideals} remains unanswered, we prefer to work with spectrally reduced ideals $I$ of a seminormed ring $A$ as opposed to general closed radical ideals. This gives us stronger control on the properties of the quotient $A/I$ as a normed ring, which is needed for our results on perfectoid Tate rings in Section 4. \begin{rmk} \label{Remark on spectral radicals}We note that Question \ref{Question on spectrally reduced ideals} has an affirmative answer for Noetherian Banach algebras over a nonarchimedean field, by Theorem \ref{Radical ideals in Noetherian Banach algebras} below. The question is also known to have an affirmative answer for perfect Banach rings of characteristic $p > 0$. Indeed, the quotient of such a Banach ring by a closed radical ideal is again a perfect Banach ring, and perfect Banach rings are automatically uniform and, a fortiori, spectrally reduced. However, even for perfectoid Tate rings in characteristic $0$, Question \ref{Question on spectrally reduced ideals} seems to be open. In fact, for perfectoid Tate rings the same question has been raised, in slightly different language, by Kedlaya in \cite{Kedlaya17}, Remark 2.9.16. \end{rmk} In Section 3 we will study the topological properties of $\TopSpec(A)$ for a general seminormed ring $A$ and in Section 4 we will turn our attention to perfectoid Tate rings. In the present section we confine ourselves to proving that for a Noetherian Banach algebra $A$ over a nonarchimedean field $K$ the topological spectrum and the usual (prime ideal) spectrum of $A$ coincide, and to deriving some consequences from this fact. In the case of affinoid algebras (both those in the sense of Tate and those in the sense of Berkovich) this is known classically (see, for example, \cite{Berkovich}, the second sentence in Proposition 2.1.4 (i)). Our proof for general Noetherian Banach algebras relies on the fact that these Banach algebras satisfy $\sCore(A) = A$. \begin{lemma} \label{core and spectral seminorm} Let $A$ be a Banach algebra without non-zero divisors of zero over a nonarchimedean field $K$. Suppose that $\Core(A,\lVert \cdot \rVert) = A$. Then \begin{enumerate}[(i)] 
\item $A$ is spectrally reduced. 
\item $\Core(A,|\cdot|_{\spc}) = A$. \end{enumerate} 
\end{lemma} 
\begin{proof} Let $x \in A \setminus \{0\}$ be arbitrary. 

(i) As $\Core(A) = A$, we have $C_x < \infty$. For all $n \in \mathbb{N}_{>0}$ we obtain the chain of inequalities $\lVert x^{n} \rVert C_{x}^{n} \geq \lVert x \rVert \lVert x^{n-1} \rVert C_{x}^{n-1} \geq \dots \geq {\lVert x \rVert}^{n}$, i.e.~we have ${\lVert x^{n} \rVert}^{1/n}C_{x} \geq \lVert x \rVert$ for all $n$ and thus $|x|_{\spc}C_x \geq \lVert x \rVert$. Hence $|x|_{\spc} = 0$ implies $\lVert x \rVert = 0$.

(ii) For all $f \in A$ we have $|f|_{\spc}|x|_{\spc} = \lim_{n \to \infty}{\lVert f^{n}\rVert}^{1/n}\lim_{n \to \infty}{\lVert x^{n}\rVert}^{1/n} = \inf_{n}(\lVert f^{n} \rVert \lVert x^{n} \rVert)^{1/n} \\ \leq \inf_{n}({C_x}^{n}\lVert (fx)^{n}\rVert)^{1/n} = C_{x}\inf_{n}{\lVert (fx)^{n}\rVert}^{1/n} = C_{x}|fx|_{\spc}$. By (i), $|x|_{\spc} \neq 0$. Hence \begin{equation*} c := {C_x}/|x|_{\spc} \end{equation*} satisfies $|f|_{\spc}\leq c|fx|_{\spc}$ for all $f \in A$. By Lemma \ref{The core and norm inequalities}, this proves (ii). \end{proof}
\begin{thm} \label{Radical ideals in Noetherian Banach algebras} Let $A$ be a Noetherian Banach algebra over a nonarchimedean field $K$. Then every radical ideal of $A$ is spectrally reduced. In particular, $\TopSpec(A) = \Spec(A)$.\end{thm}
\begin{proof}In a Noetherian Banach $K$-algebra all ideals are closed; see \cite{BGR}, Prop.~3.7.2/2. For $\mathfrak{p}\subsetneq A$ a prime ideal, the quotient $A/\mathfrak{p}$ is again Noetherian. For any $x\in A/\mathfrak{p}$ the ideal $(x)_{A/\mathfrak{p}}$ is then closed and, consequently, multiplication by $x$ is a strict map by Banach’s Open Mapping Theorem. Hence $\Core(A/\mathfrak{p})=A/\mathfrak{p}$ for every prime ideal $\mathfrak{p}$ of $A$. By Lemma 2.22(i) this entails that, for every prime ideal $\mathfrak{p}$ of $A$, the spectral seminorm on $A/\mathfrak{p}$ is a norm. This norm lifts to a bounded power-multiplicative seminorm $\phi_{\mathfrak{p}}$ on $A$ whose kernel is $\mathfrak{p}$. Therefore, for any radical ideal $J$ of $A$, we see that $J=\bigcap_{J\subseteq \mathfrak{p}}\mathfrak{p}$ coincides with the kernel of the bounded power-multiplicative seminorm $\phi_{J}=\sup_{J\subseteq \mathfrak{p}}\phi_{\mathfrak{p}}$. \end{proof}
Recall from the Notation Section, that a seminormed ring $(A,\lVert \cdot \rVert)$ is called uniform if its topology can be defined by a bounded power-multiplicative seminorm (which can then be taken to be equal to the spectral seminorm). We record the following simple lemma for later use. 
\begin{lemma} \label{Ring of power-bounded elements} If $A$ is a seminormed ring and $\lVert \cdot \rVert$ is any power-multiplicative seminorm defining the topology on $A$, then the ring of power-bounded elements $A^{\circ}$ of $A$ has the following description: \begin{equation*} A^{\circ} = \{\,f \in A \mid \lVert f \rVert \leq 1 \,\}. \end{equation*}\end{lemma}\begin{proof}Since the seminorm is power-multiplicative, $\lVert f\rVert \leq 1$ implies $\lVert f^{n}\rVert \leq 1$ for all $n$. Hence the right hand side is contained in the left hand side. Conversely, let $f$ be any power-bounded element of $A$. Then there exists some $C > 0$ such that $\lVert f\rVert^{n} = \lVert f^{n}\rVert \leq C$ for all $n$. But this entails $\lVert f\rVert \leq C^{1/n}$ for all $n$, so taking the limit as $n \to \infty$ we obtain the inequality $\lVert f\rVert \leq 1$.\end{proof}
A subset $\mathcal{S} \subset{\mathcal{M}(A)}$ is called a boundary for a normed ring $(A,\lVert \cdot \rVert)$ if for all $f \in A$ there exists $\phi \in \mathcal{S}$ such that $\phi(f) = \lVert f \rVert$. A closed boundary $\mathcal{S} \subset{\mathcal{M}(A)}$ is called Shilov boundary for $(A,\lVert \cdot \rVert)$ if it is the smallest of all closed boundaries for $(A,\lVert \cdot \rVert)$ with respect to inclusion. It was shown by Escassut and Mainetti \cite{EM1} that the Shilov boundary exists for any uniform normed algebra over a nonarchimedean field $K$. 

In \cite{Escassut1} Escassut established a close relationship between the notion of Shilov boundary and that of topological divisors of zero in a uniform normed $K$-algebra. 
\begin{thm}[Escassut] \label{Escassut's theorem} Let $A$ be a uniform normed $K$-algebra over a nonarchimedean field $K$ and let $\mathcal{S}$ denote the Shilov boundary for $(A,\lVert \cdot \rVert)$. An element $f \in A$ is a topological divisor of zero if and only if there exists $\psi \in \mathcal{S}$ such that $\psi(f) = 0$.\end{thm}
\begin{proof} See \cite{Escassut1}, Theorem 1.4.\end{proof} 
From this theorem we can deduce that some Banach $K$-algebras, including all Noetherian ones, have (in a certain sense) "enough" bounded multiplicative seminorms. This strengthens the assertion of Theorem \ref{Radical ideals in Noetherian Banach algebras}. 
\begin{prop} \label{The core and the topological spectrum} Let $A$ be a Banach algebra over a nonarchimedean field $K$ and let $\mathfrak{p} \subsetneq A$ be a closed prime ideal such that $\Core(A/\mathfrak{p}) = A/\mathfrak{p}$. Then $\mathfrak{p}\subsetneq{A}$ is the kernel of some element $\phi \in \mathcal{M}(A)$.\end{prop}
\begin{proof} Let $\mathfrak{p}\subsetneq{A}$ be a closed prime ideal as in the statement of the proposition. By hypothesis, the Banach algebra $B := A/\mathfrak{p}$ satisfies $\Core(B,\lVert \cdot \rVert) = B$, where $\lVert \cdot \rVert$ is the quotient norm induced from $A$. As usual we denote by $|\cdot|_{\spc}$ the corresponding spectral seminorm. By Lemma \ref{core and spectral seminorm} (i) the seminorm $|\cdot|_{\spc}$ is a norm and by Lemma \ref{core and spectral seminorm} (ii) we have $\Core(B,|\cdot|_{\spc}) = B$, so $(B,|\cdot|_{\spc})$ is a uniform normed $K$-algebra without nonzero topological divisors of zero. By Theorem \ref{Escassut's theorem} we conclude that there exist absolute values on $B$ continuous with respect to $|\cdot|_{\spc}$, hence also continuous with respect to $\lVert \cdot \rVert$. Any such absolute value lifts to an element $\phi \in \mathcal{M}(A)$ with $\ker(\phi) = \mathfrak{p}$.\end{proof}
\begin{cor} \label{Topological spectrum of Noetherian Banach algebras} Let $A$ be a Noetherian Banach algebra. Then every prime ideal of $A$ is of the form $\ker(\phi)$ for some $\phi \in \mathcal{M}(A)$.\end{cor}
The next lemma (in the case of Banach algebras) is due to Schikhof and was taken from \cite{Schikhof} (p. 48); it shows that we can use the condition $\Core(A) = A$ to get some control of the algebraic properties of $A$. We replicate the proof, which works in a more general setting. Note that for any Banach ring $A$ (and, in particular, for any Banach algebra $A$ over a nonarchimedean field) the group of units $A^{\times}$ is open.
\begin{lemma}[Schikhof] \label{Schikhof's lemma} Let $K$ be a nonarchimedean field and $A$ a normed algebra over $K$ such that the group of units $A^{\times}$ is open (for example, every Banach algebra has this property). Then the topological boundary $\partial A^{\times}$ of $A^{\times}$ consists of topological divisors of zero.\end{lemma} 
\begin{proof} Let $f \in \partial A^{\times}$ and let $(f_{n})_{n}$ be a sequence in $A^{\times}$ with $f_{n} \to f$. Since $A^{\times}$ is open, $f \notin A^{\times}$. Assume that $\lVert f_{n} - f \rVert < 1/{\lVert {f_{n}}^{-1}\rVert}$ for some $n$. Then $\lVert 1- {f_{n}}^{-1}f\rVert = \lVert {f_{n}}^{-1}(f_{n} - f)\rVert \leq \lVert {f_{n}}^{-1}\rVert \lVert f_{n} - f\rVert < 1$. Hence ${f_{n}}^{-1}f \in A^{\times}$ leading to $f \in A^{\times}$, a contradiction. Thus we have $\lVert f_{n} - f\rVert \geq 1/{\lVert {f_{n}}^{-1}}\rVert$ for all $n$ and, consequently, we have $\lVert {f_{n}}^{-1} \rVert \to \infty$. Choose $\lambda_{n} \in K, c_1, c_2 > 0$ such that for all $n \in \mathbb{N}$ \begin{equation*} c_1 \leq \lVert {\lambda_{n}}^{-1}{f_{n}}^{-1}\rVert \leq c_2 \end{equation*} We have $|\lambda_{n}| \geq {c_2}^{-1}\lVert {f_n}^{-1}\rVert$, thus $|\lambda_{n}| \to \infty$. For every $n \in \mathbb{N}$ set $t_n := {\lambda_{n}}^{-1}{f_n}^{-1}$. Then we obtain $\inf_{n}\lVert t_n \rVert \geq c_1 > 0$ and \begin{equation*} \lVert {t_n}f\rVert = \lVert {\lambda_{n}}^{-1}{f_{n}}^{-1}(f - f_{n}) + {\lambda_{n}}^{-1}\rVert \leq \max (c_{2}\lVert f - f_{n}\rVert , |\lambda_{n}|^{-1}) \to 0 \end{equation*} proving that $f$ is a topological divisor of zero in $A$.\end{proof}The lemma leads to a structural result for Banach algebras without non-zero topological divisors of zero. 

Recall that a ring is said to be Jacobson-semisimple if the intersection of all of its maximal ideals (the Jacobson radical) is zero. A ring is said to be Jacobson if every prime ideal is an intersection of maximal ideals.   
\begin{mydef} We call a seminormed ring $A$ topologically Jacobson if every spectrally reduced prime ideal $\mathfrak{p}\subsetneq{A}$ is an intersection of closed maximal ideals. 
\end{mydef}
If $A$ is a Banach $K$-algebra and $x \in A$, we denote by $\spc(x)$ the spectrum of $x$, i.e. $\spc(x)$ is the set of $\lambda \in K$ such that the element $x - \lambda$ is not invertible in $A$. 
\begin{thm} \label{Jacobson-semisimple Banach algebras} Every Banach algebra $A$ without non-zero topological divisors of zero is Jacobson-semisimple. If for every spectrally reduced prime ideal $\mathfrak{p} \subsetneq A$ the quotient Banach algebra has no non-zero topological divisors of zero, then $A$ is topologically Jacobson.\end{thm}
\begin{proof} The second assertion follows immediately from the first (recall that every maximal ideal in a Banach ring is closed). Let $A$ be a Banach $K$-algebra without non-zero topological divisors of zero. We want to show that $A$ is Jacobson-semisimple. To this end, consider an element $x$ in the Jacobson radical of $A$. Then $\spc(x) = \{0\}$. Indeed, assume $\lambda \in K$ is an element such that $x - \lambda \notin A^{\times}$. Then $x - \lambda$ belongs to some maximal ideal $\mathfrak{m}\subsetneq{A}$ of $A$. But since $x$ is in the Jacobson radical this means that $\lambda = x - (x-\lambda) \in \mathfrak{m}$ and thus $\lambda = 0$. Now pick any sequence $(\lambda_{n})_{n} \subset{K}$ converging to $0$. Then, since $\spc(x)=\{0\}$, we obtain a sequence $(x -\lambda_{n})_{n}$ of units in $A$ converging to the non-unit $x$. By Lemma \ref{Schikhof's lemma} $x$ must then be a topological divisor of zero. Then, by hypothesis, $x = 0$ as desired. 
\end{proof} 
\begin{cor} \label{Noetherian Banach algebras are Jacobson} Every Noetherian Banach $K$-algebra is Jacobson.\end{cor}
\begin{proof} Let $A$ be a Noetherian Banach algebra. Recall from Theorem \ref{Radical ideals in Noetherian Banach algebras} that all prime ideals of $A$ are spectrally reduced. For any prime ideal $\mathfrak{p}$ the quotient $A/\mathfrak{p}$ is a Noetherian Banach algebra whose underlying ring is an integral domain, so $A/\mathfrak{p}$ has no non-zero topological divisors of zero by Example \ref{Normed rings with full core}(1) and Lemma \ref{The core and top. divisors of zero}. By the last theorem, any prime ideal $\mathfrak{p}$ is an intersection of maximal ideals, as desired.\end{proof}
\begin{rmk} \label{Perfect Banach algebras are topologically Jacobson} Another class of topologically Jacobson Banach algebras is given by perfectoid algebras over an algebraically closed nonarchimedean field, see Corollary \ref{Perfectoid algebras over an algebraically closed field are topologically Jacobson}.\end{rmk}
We conclude this section with a toy application of the theory of Banach algebras with $\Core(A) = A$. It is well-known that the completion of a normed ring whose underlying ring is a field need not be a field, unless the norm is assumed to be multiplicative (for a simple example of this phenomenon, see \cite{Kedlaya18}, Remark 1.8). Using the notion of the core of a normed algebra we would like to collect some sufficient conditions on a normed field $A$ which would ensure that the completion $\widehat{A}$ is a field. To begin with, we prove a simple algebraic lemma. 

We call a ring $A$ a full quotient ring if every non-unit in $A$ is a zero-divisor. Equivalently, $A$ coincides with its own classical ring of fractions (in the literature full quotient rings have also come to be known as classical rings).  
\begin{lemma} \label{Full quotient rings} Let $A$ be a reduced full quotient ring with no nontrivial idempotents and only finitely many minimal prime ideals. Then $A$ is a field.\end{lemma} 
\begin{proof}Recall that in any reduced ring the set of zerodivisors coincides with the union of all minimal prime ideals. Now, if $A$ is a reduced full quotient ring as in the lemma and $\mathfrak{m}$ is a maximal ideal in $A$, then $\mathfrak{m}$ consists of zerodivisors and is thus contained in the union of the finitely many minimal prime ideals $\mathfrak{p_1}, \dots, \mathfrak{p_n}$. By the Prime Avoidance Lemma we conclude that $\mathfrak{m}$ must be equal to one of the $\mathfrak{p_i}$. Since this applies to any maximal ideal, $A$ is semilocal (has only finitely many maximal ideals) and zero-dimensional. Since $A$ is reduced, $A$ is isomorphic to a finite product of fields by the Chinese Remainder Theorem. But we assumed that $A$ has no nontrivial idempotent elements, so $A$ is a field. 
\end{proof}
We can now prove the following 
\begin{prop} \label{Completion of a normed field} Let $K$ be a nonarchimedean field, let $A$ be a normed field over $K$, i.e. $A$ is a normed $K$-algebra which is algebraically a field. Denote by $\widehat{A}$ the completion of $A$, a Banach $K$-algebra. The following are equivalent: 
\begin{enumerate}[(i)] 
\item $\widehat{A}$ is a Banach field; 
\item $\widehat{A}$ contains no non-zero topological divisors of zero (that is, $\widehat{A}$ is an integral domain and $\Core(\widehat{A}) = \widehat{A}$); 
\item $\widehat{A}$ is Noetherian, reduced and contains no nontrivial idempotents; 
\item $\widehat{A}$ is reduced, contains no nontrivial idempotents, satisfies $\Core(\widehat{A}) = \widehat{A}$ and has only a finite number of minimal prime ideals. 
\end{enumerate} 
\end{prop}
\begin{proof} Since $A$ is dense in $\widehat{A}$ we see that $\widehat{A}$ has dense group of units ${\widehat{A}}^{\times}$, so by Lemma \ref{Schikhof's lemma} statement (ii) immediately implies statement (i).
\\ 
Assume (iv). Since $\Core(\widehat{A}) = \widehat{A}$ we know (see Lemma \ref{The core and top. divisors of zero}) that every topological divisor of zero in $\widehat{A}$ is actually a divisor of zero. Again, the group of units ${\widehat{A}}^{\times}$ is dense in $\widehat{A}$. Using Lemma \ref{Schikhof's lemma} we see that $\widehat{A}$ is a full quotient ring whence we deduce that $\widehat{A}$ is a field by Lemma \ref{Full quotient rings}. 
\\ Let us now assume (iii). Then $\widehat{A}$ is a Noetherian Banach algebra and all of its ideals are closed. A fortiori all principal ideals are closed which implies $\Core(\widehat{A}) = \widehat{A}$ (see our remark immediately after Definition \ref{Def. of the core}). This shows that (iii) implies (iv). The other implications are obvious. 
\end{proof}
Recall from Section $1$ that for a seminormed ring $A$ we denote by $A^{u}$ the uniformization of $A$ which is defined as the completion of $A$ with respect to the spectral seminorm $|\cdot|_{\spc}$.  
\begin{lemma} \label{uniformization} Let $A$ be a Banach ring. Then there is a homeomorphism $\mathcal{M}(A) \simeq \mathcal{M}(A^{u})$. 
\end{lemma} 
\begin{proof} This is a special case of Lemma \ref{Berkovich spectrum}.    
\end{proof} \begin{cor} Let $A$ be a Banach field over $K$ and denote by $A^{u}$ the uniformization of $A$. The following are equivalent: 
\begin{enumerate}[(i)] 
\item $A^{u}$ satisfies $\Core(A^{u}) = A^{u}$ and has only finitely many minimal prime ideals;
\item $A^{u}$ is Noetherian; 
\item $A^{u}$ is a uniform Banach field. 
\end{enumerate} 
If the ground field $K$ is nondiscretely valued (the value semigroup $|K^{\times}|$ is dense in $\mathbb{R}_{\geq 0}$), then the above assertions are also equivalent to \\\
\linebreak
(iv) $A^{u}$ is a nonarchimedean field. 
\end{cor} 
\begin{proof} Assume (i). It is clear that every uniform Banach ring is reduced. Hence by Proposition \ref{Completion of a normed field} it suffices to show that $A^{u}$ has no nontrivial idempotents to conclude that it is a field. By Shilov's Idempotent Theorem (see \cite{Berkovich}, Thm. 7.4.1) a Banach $K$-algebra has no nontrivial idempotents if only if its Berkovich spectrum is a connected topological space. But by Lemma \ref{uniformization}, $\mathcal{M}(A^{u})$ is homeomorphic to $\mathcal{M}(A)$. This proves (i) $\Rightarrow$ (iii). Concerning the remaining implications, (iii) immediately implies (ii), and (ii) implies (i) since all ideals in a Noetherian Banach algebra are closed. The additional statement in the case when $K$ is a nondiscretely valued field follows from a result of Kedlaya (namely, \cite{Kedlaya18}, Theorem 3.7) which states that every uniform Banach field over a nondiscretely valued nonarchimedean field $K$ is necessarily a nonarchimedean field.\end{proof}

\section{Zariskian seminormed rings} 

We have seen in Section 2 that some properties classically known for Banach algebras hold also for more general normed algebras $A$ under the weaker hypothesis that the group of units $A^{\times}$ of $A$ is open (see, for example, Lemma \ref{Schikhof's lemma}). In this section we want to throw a closer look at seminormed rings satisfying this condition. We then use properties of these seminormed rings to study topological properties of the topological spectrum $\TopSpec(A)$ of a seminormed ring $A$.   

For $A$ a seminormed ring we denote by $A^{\circ \circ}$ as usual the set of topologically nilpotent elements of $A$. The following definition generalizes the definition of Zariskian adic rings in commutative algebra and the definition of Zariskian Huber rings given in \cite{Tanaka}.  
\begin{mydef} We call a seminormed ring $A$ Zariskian if the open subset $1+A^{\circ \circ} \subset A$ consists of units. 
\end{mydef} 
\begin{rmk} We observe that the property of a seminormed ring $A$ being Zariskian does not depend on a particular choice of seminorm but only on the topology of that seminormed ring. \end{rmk} 
Note that complete normed rings are Zariskian by the geometric series test and, in fact, Zariskian normed rings are defined so as to share several important properties of complete rings. The Zariskian condition can be described in many equivalent ways (our proof of the next proposition as well as of Proposition \ref{The Zariskian condition 2} below largely follows \cite{Ka-N} where analogous statements are shown for normed algebras over $\mathbb{C}$).
\begin{prop} \label{The Zariskian condition} Let $(A, \lVert \cdot \rVert)$ be a seminormed ring. Denote by $A_{<1}$ the subset of elements with seminorm $<1$. The following are equivalent: 
\begin{enumerate} 
\item The group of units $A^{\times}$ is open; 
\item $A$ is Zariskian; 
\item $1 + A_{<1} \subset A^{\times}$; 
\item Every maximal ideal of $A$ is closed; 
\item Every maximal ideal of $A$ is the kernel of some bounded multiplicative seminorm $\phi$ on $A$; 
\item For every $x \in A^{\circ \circ}$ the geometric series $\sum_{n=0}^{\infty}x^{n}$ converges in $A$.\end{enumerate}\end{prop}
\begin{proof} (4) $\Rightarrow$ (5): If $\mathfrak{m}$ is a closed maximal ideal of $A$, then the quotient seminorm on $A/\mathfrak{m}$ is a norm. By a result of Berkovich (\cite{Berkovich}, Theorem 1.2.1) there exists a bounded multiplicative seminorm on the completion of $A/\mathfrak{m}$. This seminorm restricts to a bounded absolute value on $A/\mathfrak{m}$ which lifts to a bounded multiplicative seminorm $\phi$ on $A$ with $\ker(\phi)=\mathfrak{m}$. \\\
(5) $\Rightarrow$ (3): Let $\phi$ be a bounded multiplicative seminorm on $A$. Then, for all $x \in A_{<1}$, we have $\phi(1 + x)=1$. In particular, $1 + x \notin \ker(\phi)$. So, $1 + x$ does not belong to any maximal ideal of $A$. That is, the set $1 + A_{<1}$ consists of units. \\\
(3) $\Rightarrow$ (2): Let $x \in A^{\circ \circ}$. Then $x^{k} \in A_{<1}$ for some $k \in \mathbb{N}_{>0}$. By assumption $1 - x^{k}$ is a unit. But by a telescoping argument 
\begin{equation*} (1 - x)(1 + x + x^{2} + \dots + x^{k-1}) = 1 - x^{k}, 
\end{equation*}
so $1 - x$ is a unit. \\\
(2) $\Rightarrow$ (1): The group of units in a topological ring is open whenever it has nonempty interior. \\\ 
(1) $\Rightarrow$ (4): If some maximal ideal of $A$ is not closed, then it is dense (since its closure is an ideal). But then the group of units is not open. 

It remains to prove the equivalence of (6) and any of the other conditions. We prove the equivalence of (6) and (2). For the implication $(2) \Rightarrow (6)$ assume that $A$ is Zariskian and $x \in A^{\circ \circ}$. By assumption, $1 - x \in A^{\times}$. A telescoping argument yields \begin{equation*} 1 + x + x^{2} + \dots + x^{n} = (1-x)^{-1}(1-x^{n+1}). \end{equation*}Thus the left hand side converges to $(1-x)^{-1}$ as $n \to \infty$, the element $x$ being topologically nilpotent. For the converse implication, suppose that for every $x \in A^{\circ \circ}$ the series $\sum_{n=0}^{\infty}x^{n}$ converges in $A$ and let $x \in A^{\circ \circ}$. Then \begin{equation*} (1-x)\sum_{n=0}^{\infty}x^{n} = \sum_{n=0}^{\infty}x^{n} - \sum_{n=1}^{\infty}x^{n} = 1, \end{equation*}so $1-x$ is invertible in $A$, as desired.\end{proof}We note the following elementary consequence of Proposition \ref{The Zariskian condition}.
\begin{cor}Let $A$ be a spectrally reduced Zariskian normed ring. Then the normed ring $(A, |\cdot|_{\spc})$ is again Zariskian.\end{cor}
\begin{proof} If $\phi \in \mathcal{M}(A)$, then $\phi(x) \leq |x|_{\spc}$ for all $x$, i.e. $\phi \in \mathcal{M}(A, |\cdot|_{\spc})$. Now apply the equivalence of (2) and (5) from the proposition.\end{proof}
\begin{cor}Let $A$ be a Huber ring with pair of definition $(A_{0}, I)$ equipped with a structure of seminormed ring as in \cite{Kedlaya17}, Remark 1.5.4. Then $A$ is Zariskian if and only if $A_{0}$ is Zariskian with respect to the $I$-adic seminorm.\end{cor}
\begin{proof}By the definition of the seminorm on $A$ an element $x \in A$ is of seminorm $<1$ if and only if $x \in I$. The same is true for the $I$-adic seminorm on $A_{0}$. Thus the corollary follows from the equivalence of properties (2) and (3) in Proposition \ref{The Zariskian condition}.\end{proof}
For normed $K$-algebras over a nonarchimedean field $K$ besides the equivalences of Proposition \ref{The Zariskian condition} one can characterize the Zariskian property as follows.  
\begin{prop} \label{The Zariskian condition 2} Let $K$ be a nonarchimedean field and $A$ a normed $K$-algebra. Then the following are equivalent: 
\begin{enumerate} 
\item $A$ is Zariskian;  
\item Every element in the boundary $\partial A^{\times}$ of $A^{\times}$ is a topological divisor of zero; 
\item  The function $\tau : A \to \mathbb{R}_{\geq 0}$ defined by $\tau(x) := \sup\{\, |\lambda| \mid \lambda \in \spc(x)\cup \{0\} \,\}$ satisfies $\tau(x) \leq |x|_{\spc}$ for all $x$; 
\item The function $\tau$ satisfies $\tau(x) \leq \lVert x \rVert$ for all $x$; 
\item There exists some $c > 0$ such that $\tau(x) \leq c\lVert x \rVert$ for all $x$; 
\item If $x \in A$ satisfies $\lVert x \rVert < 1$, then $1-x^{n} \in A^{\times}$ for some $n \in \mathbb{N}_{>0}$; 
\item If $x \in A$ has the property that the series $\sum_{n=0}^{\infty}{\lVert x \rVert}^{n}$ converges, then the series $\sum_{n=0}^{\infty}x^{n}$ converges in $A$;   
\end{enumerate} \end{prop}
\begin{proof} Clearly $(3) \Rightarrow (4) \Rightarrow (5)$. \\\
$(2) \Rightarrow (1)$: A unit in a normed ring cannot be a topological divisor of zero, and a subset of a topological space which is disjoint from its own boundary is open. Hence the condition (2) implies that the group of units $A^{\times}$ in $A$ is open. It follows from Proposition \ref{The Zariskian condition} that $A$ is Zariskian. \\\
$(1) \Rightarrow (2)$: As by Proposition \ref{The Zariskian condition} the group of units is open, this follows from Lemma \ref{Schikhof's lemma}. \\\
$(5) \Rightarrow (6)$: Given $x \in A$ with $\lVert x \rVert < 1$ pick $n \in \mathbb{N}_{>0}$ so large that ${\lVert x \rVert}^{n} < 1/c$. For $\lambda \in \spc(x^{n})$ we have $|\lambda| \leq \tau(x^{n}) \leq c\lVert x^{n}\rVert \leq c{\lVert x \rVert}^{n} < 1$. In particular, $1 \notin \spc(x^{n})$, i.e. $1-x^{n} \in A^{\times}$. \\\ 
$(6) \Rightarrow (7)$:  Let $x \in A$ such that the series $\sum_{n=0}^{\infty}{\lVert x \rVert}^{n}$ converges. Then $x$ is necessarily topologically nilpotent and $\lVert x^{k}\rVert < 1$ for some $k$. Then we know that for some (possibly larger) $k$ the element $1-x^{k}$ is a unit. By a telescoping argument 
\begin{equation*} (1 - x)(1 + x + x^{2} + \dots + x^{k-1}) = 1 - x^{k}, 
\end{equation*}
so $1 - x$ is invertible. Another telescoping argument yields 
\begin{equation*} 1 + x + x^{2} + \dots + x^{n} = {(1 - x)}^{-1}(1 - x^{n+1}) 
\end{equation*} 
and the expression on the right hand side converges to $(1-x)^{-1}$ since $x$ is topologically nilpotent.
$(7) \Rightarrow (1)$: Property (7) clearly implies property (6) in Proposition \ref{The Zariskian condition}, hence (1) follows from Proposition \ref{The Zariskian condition}. 
$(1) \Rightarrow (3)$: Fix some element $x \in A$. Let $\lambda \in \spc(x)$; then $x - \lambda$ belongs to some maximal ideal $\mathfrak{m}$ of $A$. By Proposition \ref{The Zariskian condition} there exists a bounded absolute value $|\cdot|$ on the quotient $A/\mathfrak{m}$. If $\chi: A \to A/\mathfrak{m}$ is the natural quotient map, then $|\chi(x)| = |\lambda|$. As $|\chi|$ defines an element of $\mathcal{M}(A)$, we obtain \begin{equation*} \tau(x) \leq \sup_{\phi \in \mathcal{M}(A)}\phi(x) = |x|_{\spc}.\end{equation*}\end{proof}
\begin{rmk} In the case of normed algebras over $\mathbb{C}$, the algebras we call Zariskian have come to be known as Q-algebras, a terminology introduced by Kaplansky in \cite{Kaplansky}. We have opted for the terminology of \cite{Tanaka} since it is more in step with the case of adic rings.\end{rmk}   
The different characterisations of Zariskian normed $K$-algebras can be used to prove automatic continuity results as exemplified by the next proposition. It generalises Theorem 15.3 in Escassut's book \cite{Escassut2} to (possibly incomplete) Zariskian normed algebras. The symbol $\tau$ has the same meaning as in the previous proposition.
\begin{prop} \label{Spectral radius and automatic continuity} Let $B$ be a normed $K$-algebra whose spectral seminorm $|\cdot|_{\spc}^{B}$ satisfies $|x|_{\spc}^{B} = \tau(x)$ for all $x \in B$ (note that $B$ is automatically Zariskian). Then any $K$-algebra homomorphism $\phi: A \to B$ where $A$ is a Zariskian normed algebra with spectral semi-norm $|\cdot|_{\spc}^{A}$ satisfies 
\begin{equation*} |\phi(x)|_{\spc}^{B} \leq |x|_{\spc}^{A}.\end{equation*}
\end{prop}
\begin{proof} Since $\phi$ maps invertible elements to invertible elements, $\tau(\phi(x)) \leq \tau(x)$ for all $x \in A$. Then, for any $x \in A$, $|\phi(x)|_{\spc}^{B} = \tau(\phi(x)) \leq \tau(x) \leq |x|_{\spc}^{A}$, where the last inequality follows from $A$ being Zariskian by Proposition \ref{The Zariskian condition 2}.\end{proof}
The Zariskian condition is stable under some basic operations on seminormed rings.
\begin{prop} \label{Zariskian condition and quotients} Let $A$ be a Zariskian seminormed ring and let $I \subsetneq A$ be a closed ideal. Then the normed ring $A/I$ is Zariskian.\end{prop}
\begin{proof} For a maximal ideal $\mathfrak{m} \subsetneq A/I$ consider the pre-image $\mathfrak{m}'$ of $\mathfrak{m}$ in $A$. By Proposition \ref{The Zariskian condition} we can choose $\phi \in \mathcal{M}(A)$ with $\mathfrak{m}' = \ker(\phi)$. Then $\widetilde{\phi}(x + I) = \phi(x), x \in A$, defines an element $\widetilde{\phi}$ of $\mathcal{M}(A/I)$ with $\ker(\widetilde{\phi}) = \mathfrak{m}$. 
\end{proof} Recall from \cite{FK}, Ch.~0, Section 7.3(b), that, given an adic ring $A_{0}$ with ideal of definition $I$, the associated Zariski ring or Zariskisation $A_{0}^{\Zar}$ of $A_{0}$ is the localization of $A$ at the multiplicative subset $1 + I$. Now observe that, in any nonarchimedean seminormed ring $A$, the set $1 + A_{<1}$ is a multiplicative subset of $A$. The following proposition was proved in the case of Zariskian Huber rings by Tanaka (\cite{Tanaka}, Theorem 3.12 and Theorem 3.15).\begin{prop} \label{Zariskisation} The inclusion of the category of Zariskian seminormed rings and bounded ring homomorphisms into the category of all seminormed rings and bounded ring homomorphisms possesses a left-adjoint given by $(A, \lVert \cdot \rVert) \mapsto (A^{\Zar}, \lVert \cdot \rVert_{\Zar})$ where \begin{equation*} A^{\Zar} := (1 + A_{<1})^{-1}A \end{equation*} and the seminorm $\lVert \cdot \rVert_{\Zar}$ on $A^{\Zar}$ is defined by \begin{equation*} \lVert \frac{a}{s}\rVert_{\Zar} = \frac{\lVert a \rVert}{\lVert s \rVert} = \lVert a \rVert. \end{equation*}
\end{prop}
\begin{proof} We first show that the function $\lVert \cdot \rVert_{\Zar}$ on $A^{\Zar}$ is well-defined, i.e.~that it does not depend on the representation of an element of $A^{\Zar}$ by a fraction. Hence let $a, b \in A$ and $t, s \in 1 + A_{<1}$ be elements such that \begin{equation*} \frac{a}{s} = \frac{b}{t} \end{equation*} in $A^{\Zar}$. This means that for some $t' \in 1 + A_{<1}$ we have the relation $t'ta = t'sb$ in $A$. Since $1 + A_{<1}$ is a multiplicative subset of $A$ we can write $t't = 1+x$ and $t's = 1+y$ for some $x, y \in A_{<1}$. But then \begin{equation*} \lVert xa \rVert \leq \lVert x \rVert \lVert a \rVert < \lVert a \rVert \end{equation*}and, consequently, \begin{equation*} \lVert t'ta \rVert = \lVert a + xa \rVert = \lVert a \rVert. \end{equation*}Similarly, $\lVert t'sb\rVert = \lVert b + yb \rVert = \lVert b \rVert$. This shows that $\lVert a \rVert = \lVert b \rVert$ and the function $\lVert \cdot \rVert_{\Zar}$ is well-defined. 

It is now readily seen that $\lVert \cdot \rVert_{\Zar}$ is indeed a seminorm on the ring $A^{\Zar}$, for any seminormed ring $A$. We want to verify that, for any seminormed ring $A$, the seminormed ring $A^{\Zar}$ is Zariskian. Let $y \in (A^{\Zar})_{<1}$. Write $y$ in the form $y = \frac{a_{1}}{1 + a_{2}}$ with $a_{1} \in A, a_{2} \in A_{<1}$. Then $\lVert a_1 \rVert = \lVert y \rVert_{\Zar} < 1$. Hence $1 + y = 1+ \frac{a_{1}}{1 + a_{2}} = \frac{1 + a_{1} + a_{2}}{1 + a_{2}} \in (1+A_{<1})^{-1}(1 + A_{<1})$ whence we see that $1+ y$ is a unit $A^{\Zar}$. It follows by Proposition \ref{The Zariskian condition}~that $A^{\Zar}$ is Zariskian. 

The natural map $A \to A^{\Zar}$ is bounded (by definition, it is even submetric), so the image of $A^{\circ \circ}$ in $A^{\Zar}$ is contained in $(A^{\Zar})^{\circ \circ}$. In particular, the image of the set $1 + A^{\circ \circ} \subset A$ consists of units in $A^{\Zar}$. This means that the natural homomorphism $A^{\Zar} = (1+ A_{<1})^{-1}A \to (1 + A^{\circ \circ})^{-1}A$ is an isomorphism.   

Now let $\varphi: A \to B$ be a bounded homomorphism from a seminormed ring $A$ to a Zariskian seminormed ring $B$. Since $\varphi$ is continuous, it sends $1 + A^{\circ \circ}$ into $1 + B^{\circ \circ}$. Since all elements of $1 + B^{\circ \circ}$ are invertible in $B$, the homomorphism $\varphi$ factors uniquely through the natural map $A \to A^{\Zar}$. This shows that the assignment $A \mapsto A^{\Zar}$ is a functor which is left-adjoint to the full embedding of the category of Zariskian seminormed rings into the category of seminormed rings.\end{proof}
For any seminormed ring we call the seminormed ring $A^{\Zar}$ from Proposition \ref{Zariskisation} the Zariskisation of $A$. We also use this name for the functor $A \mapsto A^{\Zar}$.\begin{cor} \label{Zariskian condition and filtered colimits} Any filtered colimit (in the category of seminormed rings and bounded ring homomorphisms) of Zariskian seminormed rings is Zariskian.\end{cor}
\begin{proof}By the above proposition the Zariskisation functor $A \mapsto A^{\Zar}$ has a right-adjoint, hence it commutes with filtered colimits.\end{proof}
As we mentioned before, all Banach rings (complete normed rings) are Zariskian. Corollary \ref{Zariskian condition and filtered colimits} supplies us with many examples of Zariskian normed rings which are not complete. For example, for a perfectoid field $K$ with pseudo-uniformizer $\varpi \in K$ the $K$-algebra \begin{equation*} (T_{n})^{1/p^{\infty}} = K \langle X_{1}, \dots, X_{n} \rangle [X_{1}^{1/p^{\infty}}, \dots, X_{n}^{1/p^{\infty}}], \end{equation*} equipped with the Gauss norm $\lVert \sum_{\nu \in (\mathbb{Z}[1/p]_{\geq 0})^{n}}a_{\nu}X^{\nu} \rVert = \max_{\nu}|a_{\nu}|$, is the direct limit of the Banach algebras \begin{equation*} (T_{n})^{1/p^{m}} = K \langle X_{1}, \dots, X_{n} \rangle [X_{1}^{1/p^{m}}, \dots, X_{n}^{1/p^{m}}] \end{equation*} (each of which is in fact isometrically isomorphic to $T_{n} = K \langle X_{1}, \dots, X_{n} \rangle$ via the map which raises each variable to its $p^{m}$-th power). Hence it is Zariskian by virtue of Corollary \ref{Zariskian condition and filtered colimits}. However, $(T_{n})^{1/p^{\infty}}$ is not complete since the sequence \begin{equation*} (\sum_{k= 0}^{m}\varpi^{k}X_{1}^{1/p^{k}})_{m \in \mathbb{N}} \subset (T_{n})^{1/p^{\infty}} \end{equation*} is a Cauchy sequence without limit in $(T_{n})^{1/p^{\infty}}$.

We now turn to the discussion of topological properties of the topological spectrum of a seminormed ring. Therein the notion of Zariskisation introduced above plays a crucial role. We begin with a simple lemma which ensures that the topological spectrum of a seminormed ring is non-empty.
\begin{lemma}\label{Existence of spectrally reduced ideals} If an ideal $I \subsetneq A$ in a seminormed ring is not dense, then it is contained in a spectrally reduced prime ideal. In particular, the topological spectrum of any seminormed ring $A$ is non-empty.\end{lemma}
\begin{proof}The first statement follows from Lemma \ref{Berkovich spectrum of a quotient} and non-emptiness of the Berkovich spectrum of the quotient $A/I$ (Lemma \ref{Berkovich spectrum}). The second statement follows by applying the first to the zero ideal of $A$ since the zero ideal in a seminormed ring is never dense.\end{proof}
Just as in commutative algebra the notion of a radical ideal in a ring leads to the notion of the radical of an arbitrary ideal, we can introduce a notion of 'spectral radical' for ideals in a seminormed ring. This will prove useful not only in this section but also while studying the tilting map for spectrally reduced ideals of a perfectoid Tate ring in Section 4. 
\begin{mydef}[Spectral radical] \label{Spectral radical} Let $A$ be a seminormed ring and $I \subsetneq A$ any ideal which is not dense in $A$. We define the spectral radical $I_{\spc}$ of $I$ in $A$ as the intersection of all spectrally reduced ideals of $A$ containing $I$. Equivalently, \begin{equation*} I_{\spc} = \bigcap_{\substack{\phi \in \mathcal{M}(A) \\ \ker(\phi) \supseteq I}}\ker(\phi). \end{equation*}
\end{mydef}
\begin{rmk}Lemma \ref{Existence of spectrally reduced ideals} ensures that the notion of the spectral radical is well-defined for any non-dense ideal in a seminormed ring. Note that by Proposition \ref{The Zariskian condition} Zariskian seminormed rings have no dense ideals, so the spectral radical is well-defined for every ideal in a Zariskian seminormed ring. \end{rmk}
\begin{lemma}\label{Topological spectrum and Zariskisation} For any seminormed ring there is an inclusion-preserving homeomorphism \begin{equation*} \TopSpec(A) \simeq \TopSpec(A^{\Zar}). \end{equation*}\end{lemma}
\begin{proof}It suffices to construct an inclusion-preserving surjection from the set of spectrally reduced ideals of $A$ onto the set of spectrally reduced ideals of $A^{\Zar}$ which takes prime ideals to prime ideals and whose restriction to $\TopSpec(A)$ is injective. To see this, suppose that such a surjection $\pi$ has been constructed. A closed subset of $\TopSpec(A)$ is of the form \begin{equation*} \mathcal{V}_{A}(I) = \{\, \mathfrak{p} \in \TopSpec(A) \mid \mathfrak{p} \supseteq I \,\} \end{equation*} and analogously for $A^{\Zar}$. If $\mathcal{V}_{A}(I)$ is non-empty (note that by Lemma \ref{Existence of spectrally reduced ideals} this is the case if and only if $I$ is not dense in $A$), the spectral radical $I_{\spc}$ of $I$ is well-defined and we have $\mathcal{V}_{A}(I) = \mathcal{V}_{A}(I_{\spc})$. But $I_{\spc}$ is spectrally reduced, so we can apply the map $\pi$ to it and see that a spectrally reduced prime ideal $\mathfrak{p} \subsetneq A$ contains $I_{\spc}$ if and only if $\pi(\mathfrak{p})$ contains $\pi(I_{\spc})$. That is, $\pi(\mathcal{V}_{A}(I_{\spc})) = \mathcal{V}_{A^{\Zar}}(\pi(I_{\spc}))$. This shows that the restriction of $\pi$ to $\TopSpec(A)$ is a closed map. The fact that $\pi$ is also a continuous map is shown in a completely analogous way.   

We now turn to the construction of a surjection $\pi$ as above. We claim that the map from the set of spectrally reduced ideals of $A$ to the set of spectrally reduced ideals of $A^{\Zar}$ which is given by $J \mapsto JA^{\Zar}$ has the required properties. First of all, this map is well-defined since no spectrally reduced ideal of $A$ can intersect the set $1 + A_{<1}$ (every $\phi \in \mathcal{M}(A)$ satisfies $\phi(f) < 1$ and hence $\phi(1 + f) = 1$ for any $f \in A_{<1}$). The map $J \mapsto JA^{\Zar}$ is clearly inclusion-preserving, takes prime ideals to prime ideals and its restriction to $\TopSpec(A)$ is injective, since $\mathfrak{p}A^{\Zar} \cap A = \mathfrak{p}$ if $\mathfrak{p}$ is a spectrally reduced prime ideal. It remains to prove that the map is surjective. Let $I$ be a spectrally reduced ideal of $A^{\Zar}$, let $J = I \cap A$. Then $I = JA^{\Zar}$ by the usual properties of localization and $J$ is spectrally reduced since the restriction to $A$ of any bounded power-multiplicative seminorm $\phi$ on $A^{\Zar}$ with $\ker(\phi) = I$ must have kernel $J$.
\end{proof}
\begin{cor} \label{Maximal spectrally reduced ideals} In any seminormed ring $A$, every spectrally reduced ideal is contained in a maximal spectrally reduced ideal.\end{cor}
\begin{proof} Every ideal of $A^{\Zar}$ is contained in a maximal ideal and every maximal ideal of $A^{\Zar}$ is spectrally reduced (Proposition \ref{The Zariskian condition}), so the claim follows from the above lemma.\end{proof}
We now give a description of closed irreducible subsets of $\TopSpec(A)$ which parallels the usual description of closed irreducible subsets of the usual prime ideal spectrum.  
\begin{lemma} \label{Closed irreducible subsets} Let $A$ be a seminormed ring. A closed subset of $\TopSpec(A)$ is irreducible if and only if it is of the form \begin{equation*} \mathcal{V}(\mathfrak{p}) = \{\, \mathfrak{q} \in \TopSpec(A) \mid \mathfrak{p} \subseteq \mathfrak{q} \,\} \end{equation*}for some spectrally reduced prime ideal $\mathfrak{p}$ of $A$. \end{lemma}
\begin{rmk} Note that all closed irreducible subsets of $\TopSpec(A)$ are trivially of the form $\mathcal{V}(\mathfrak{p})$ for some prime ideal $\mathfrak{p} \subsetneq A$, the topology on $\TopSpec(A)$ being the restriction of the Zariski topology of $\Spec(A)$. However, the spectral radical of an arbitrary prime ideal need not a priori be a prime ideal, so we have to prove that our prime ideal $\mathfrak{p}$ can be chosen to be spectrally reduced. \end{rmk}
\begin{proof}[Proof of Lemma \ref{Closed irreducible subsets}] Every closed subset of $\TopSpec(A)$ is of the form $\mathcal{V}(I)$ for some spectrally reduced ideal $I$. We want to show $\mathcal{V}(I)$ is irreducible if and only if $I$ is a spectrally reduced prime ideal. Suppose that $I$ is not a prime ideal. Then there exist ideals $I_1, I_2$ of $A$ with $I_{1}\cdot I_{2} \subseteq I$ but $I_{1} \not \subseteq I$ and $I_{2} \not \subseteq I$. The inclusion $I_{1}\cdot I_{2} \subseteq I$ means $\mathcal{V}(I) = \mathcal{V}(I_{1}) \cup \mathcal{V}(I_{2})$. Since every spectrally reduced ideal is an intersection of spectrally reduced prime ideals, $I_{i} \not \subseteq I$ (for $i = 1, 2$) means that for each $i = 1,2$ there exists a spectrally reduced prime ideal which contains $I$ but does not contain $I_{i}$. It follows that $\mathcal{V}(I) \neq \mathcal{V}(I_{i})$ for any $i = 1, 2$. This shows that $\mathcal{V}(I)$ can be written as a union of two proper closed subsets $\mathcal{V}(I_1), \mathcal{V}(I_2)$, so $\mathcal{V}(I)$ is not irreducible. Conversely, let $\mathfrak{p}$ be a spectrally reduced prime ideal of $A$ and let $I_{1}, I_{2}$ be ideals of $A$ such that $\mathcal{V}(\mathfrak{p}) = \mathcal{V}(I_{1}) \cup \mathcal{V}(I_{2})$. Since $\mathfrak{p}$ is itself a spectrally reduced prime ideal, $\mathfrak{p} \in \mathcal{V}(\mathfrak{p})$, so either $I_{1} \subseteq \mathfrak{p}$ and $\mathcal{V}(\mathfrak{p}) = \mathcal{V}(I_{1})$ or $I_{2} \subseteq \mathfrak{p}$ and $\mathcal{V}(\mathfrak{p}) = \mathcal{V}(I_{2})$, showing that the closed subset $\mathcal{V}(\mathfrak{p})$ is indeed irreducible.  
\end{proof} Recall that a topological space is called sober if every irreducible closed subset has a unique generic point.   
\begin{cor} \label{Sobriety} The topological spectrum $\TopSpec(A)$ of a seminormed ring $A$ is a sober topological space. 
\end{cor}
\begin{proof} By Lemma \ref{Closed irreducible subsets} every closed irreducible subset $Z$ of $\TopSpec(A)$ is of the form $Z = \mathcal{V}(\mathfrak{p})$ so that $\mathfrak{p}$ is the unique generic point of $Z$.\end{proof}
The following lemma tells us that, similarly to the usual prime ideal spectrum, the topological spectrum of a seminormed ring can be decomposed into irreducible components.   
\begin{lemma} \label{Minimal spectrally reduced primes} Let $(A, \lVert \cdot \rVert)$ be a seminormed ring. Every spectrally reduced prime ideal of $A$ contains a minimal spectrally reduced prime ideal of $A$.\end{lemma}
\begin{proof} Let $\mathfrak{p}_{0} \supsetneq \mathfrak{p}_{1} \supsetneq \dots \supsetneq \mathfrak{p}_{n} \supsetneq \dots$ be a strictly descending chain of spectrally reduced prime ideals in $A$ which are contained in a fixed spectrally reduced prime ideal $\mathfrak{p}$. Let $\alpha_{i}$ be a bounded power-multiplicative seminorm on $A$ with $\ker(\alpha_{i}) = \mathfrak{p}_{i}$. Consider the bounded power-multiplicative seminorm $\alpha$ defined by $\alpha(f) := \sup_{i}\alpha_{i}(f)$ (the supremum is always finite since $\alpha_{i}(f) \leq \lVert f \rVert$ for every $i$ and every $f \in A$). Then \begin{equation*} \ker(\alpha) = \bigcap_{i}\mathfrak{p}_{i}. \end{equation*} Let $f, g \in A$ be elements of $A$ such that $fg \in \bigcap_{i}\mathfrak{p}_{i} =: \mathfrak{p}$ and $g \notin \mathfrak{p}$. Choose $i$ such that $g \notin \mathfrak{p}_{i}$. Then for every $j \geq i$ we have $f \in \mathfrak{p}_{j}$, so $f \in \mathfrak{p}$. This shows that $\mathfrak{p} = \ker(\alpha)$ is a spectrally reduced prime ideal and hence it shows that every strictly descending chain of spectrally reduced prime ideals contained in $\mathfrak{p}$ admits a lower bound. We conclude by applying Zorn's Lemma 'downwards'.\end{proof}
\begin{thm} \label{The topological spectrum is quasi-compact} The topological spectrum $\TopSpec(A)$ of any seminormed ring $A$ is a quasi-compact topological space.\end{thm}
\begin{proof} By Lemma \ref{Topological spectrum and Zariskisation}~we may assume that our seminormed ring $A$ is Zariskian. To prove quasi-compactness of $X = \TopSpec(A)$ it suffices to consider open covers of $X$ by principal open subsets \begin{equation*} D_{\spc}(f) = \{\, \mathfrak{p} \in \TopSpec(A) \mid f \notin \mathfrak{p} \,\} \end{equation*}since these form a basis for the topology on $X$. Let $(D_{\spc}(f_{i}))_{i \in I}$ be such a cover for some (possibly infinite) index set $I$. This means that the ideal of $A$ generated by the elements $f_{i}, i \in I$, is not contained in any spectrally reduced prime ideal of $A$. But $A$ is assumed to be Zariskian, so every maximal ideal of $A$ is spectrally reduced. Hence the ideal generated by $f_{i}, i \in I$, is not contained in any maximal ideal, i.e.~the elements $f_{i}, i \in I$, generate the unit ideal of $A$. Then we can find a finite subset $I'$ of $I$ such that $f_{i}, i \in I'$, generate the unit ideal. This gives us a finite subcover of the open cover $(D_{\spc}(f_{i}))_{i \in I}$ of $X$, establishing that $X$ is indeed quasi-compact.\end{proof}
In commutative algebra it is a common theme to describe certain properties of a ring $A$ as properties of the topological space $\Spec(A)$. Recall that a topological space is called a Jacobson space if its closed points are dense in every closed subset. For example, a ring $A$ is Jacobson if and only if $\Spec(A)$ is a Jacobson topological space (\cite{Stacks}, Part 1, Lemma 10.34.4), and every radical ideal of $A$ is the radical of a finitely generated ideal if and only if the topological space $\Spec(A)$ is Noetherian (see \cite{Ohm-Pendleton}, Proposition 2.1). Analogous assertions can be proved for the topological spectrum of a Zariskian seminormed ring.
\begin{prop} \label{Topological Jacobson property and topological spectrum} If a seminormed ring $A$ is topologically Jacobson, then $\TopSpec(A)$ is a Jacobson topological space. For Zariskian seminormed rings the converse is also true.\end{prop}
\begin{proof} Suppose that $A$ is topologically Jacobson. Let $Z \subset \TopSpec(A)$ be a closed subset. We have to show that the set of closed points in $Z$ is dense in $Z$. Let $f \in A$ such that $D_{\spc}(f) \cap Z \neq \emptyset$. Write $Z = \mathcal{V}(I)$ for a spectrally reduced ideal $I$ (if $I$ is not spectrally reduced, replace $I$ with its spectral radical). We see that $f \notin I$. Since by assumption $A$ is topologically Jacobson, the spectrally reduced ideal $I$ is the intersection of closed maximal ideals containing it. Hence $f \notin I$ implies that there exists some closed maximal ideal $\mathfrak{m}$ such that $I \subsetneq \mathfrak{m}$ but $f \notin \mathfrak{m}$. Since the Berkovich spectrum of the normed ring $A/\mathfrak{m}$ is non-empty (Lemma \ref{Berkovich spectrum}), the maximal ideal $\mathfrak{m}$ is spectrally reduced (by \cite{Berkovich}, Theorem 1.2.1). Hence $\mathfrak{m} \in D_{\spc}(f) \cap Z$, showing that $D_{\spc}(f) \cap Z$ contains a closed point of $Z$. 

Conversely, suppose that $A$ is Zariskian and $\TopSpec(A)$ is Jacobson. Let $\mathfrak{p} \subsetneq A$ be a spectrally reduced prime ideal of $A$. Let $J = \bigcap_{\mathfrak{p} \subset \mathfrak{m}}\mathfrak{m}$ be the intersection of the maximal ideals containing $\mathfrak{p}$. By Proposition \ref{The Zariskian condition}, $J$ is an intersection of spectrally reduced prime ideals, $\mathcal{V}(J) \subseteq \mathcal{V}(\mathfrak{p})$ and $\mathcal{V}(J)$ is the smallest closed subset of $\mathcal{V}(\mathfrak{p})$ containing all closed points of $\mathcal{V}(\mathfrak{p})$. By assumption $\mathcal{V}(J) = \mathcal{V}(\mathfrak{p})$. But this means $\mathfrak{p} \supseteq J$ whence $J = \mathfrak{p}$. We conclude that $A$ is topologically Jacobson.\end{proof}
The last result of this section is a partial analogue of Proposition 2.1 in \cite{Ohm-Pendleton}.\begin{prop}\label{Noetherian topological spectrum} Let $A$ be a Zariskian seminormed ring. If the topological space $\TopSpec(A)$ is Noetherian, then every spectrally reduced ideal is the spectral radical of a finitely generated ideal.\end{prop}
\begin{proof} If the topological space $X$ is Noetherian, then every open subset of $X$ is quasi-compact. That is, the subset $X \setminus \mathcal{V}(I)$ is quasi-compact for any spectrally reduced ideal $I$ of $A$. Since the principal open subsets $D_{\spc}(f)$ (for $f \in A$) form a basis for the topology of $X$, this means that there exist finitely many elements $f_1, \dots, f_n \in A$ with $X \setminus \mathcal{V}(I) = \bigcup_{i=1}^{n}D_{\spc}(f_i)$. In other words, a spectrally reduced ideal $\mathfrak{p}$ of $A$ contains $I$ if and only if it contains all of the elements $f_{1}, \dots, f_{n}$. But this means precisely that $I$ is equal to the spectral radical of $(f_1, \dots, f_n)_{A}$.\end{proof}

\section{Perfectoid rings and spectrally reduced ideals} 

In the sequel fix a prime number $p$. A perfectoid Tate ring (in the sense of Fontaine) is a uniform Banach ring $R$ containing a topologically nilpotent unit $\varpi$ such that $p \in {\varpi}^{p}R^{\circ}$ and the Frobenius $\phi: R^{\circ}/{\varpi} \to R^{\circ}/{\varpi}^{p}$ is surjective. If $K$ is a perfectoid Tate ring which is a field, it is automatically a nonarchimedean field by Theorem 4.2 in \cite{Kedlaya18}, and $K$ is called a perfectoid field. For $R$ a uniform Banach algebra over some perfectoid field $K$ the perfectoid condition is equivalent to: The Frobenius map $\phi : R^{\circ}/pR^{\circ} \to R^{\circ}/pR^{\circ}$ on the mod-$p$ reduction of the ring of power-bounded elements $R^{\circ}$ is surjective. If moreover $K$ is a perfectoid field of characteristic $p$, then a Banach algebra $R$ over $K$ is perfectoid if and only if it is perfect (\cite{Scholze}, Proposition 5.9). In what follows we think of any perfectoid Tate ring as equipped with a power-multiplicative norm defining its topology.    

For any perfectoid Tate ring $R$ we can construct the tilt $R^{\flat}$ of $R$ which is a perfect uniform Banach ring of characteristic $p$. For any fixed perfectoid Tate ring $R$ we have an equivalence of categories $A \mapsto A^{\flat}, \varphi \mapsto \varphi^{\flat}$, between the category of perfectoid $R$-algebras and the category of perfectoid $R^{\flat}$-algebras. The definition of $R^{\flat}$ (resp. $A^{\flat}$) in terms of sequences of elements of $R$ (resp. of $A$) has been recalled in the introduction. The norm of an element $x = (x^{(n)})_{n} \in A^{\flat}$ is defined as $\lVert x \rVert := \lVert x^{(0)}\rVert$. On the level of morphisms, we associate to a bounded homomorphism $\varphi: A \to B$ of perfectoid $R$-algebras a bounded homomorphism of $R^{\flat}$-algebras $\varphi^{\flat}: A^{\flat} \to B^{\flat}$  defined by
\begin{equation*}\varphi^{\flat}(f) := (\varphi(f^{(n)}))_{n}, \ f = (f^{(n)})_{n \in \mathbb{N}_{0}} \in A^{\flat}.\end{equation*}
As in the introduction $f \mapsto f^{\#}$ is the natural multiplicative map $A^{\flat} \to A$: If $f \in A^{\flat}$ is given by the sequence $(f^{(n)})_{n \geq 0}$ of elements of $A$, then $f^{\#}$ is defined as the $0$-th term of that sequence. It turns out that the image of the map $f \mapsto f^{\#}$ generates a dense subring of $A$ (see \cite{Kedlaya-Liu}, Proposition 3.6.25(b)), so a morphism $\varphi$ can indeed be recovered from its tilt $\varphi^{\flat}$. There is also a correspondence $R^{+} \mapsto R^{\flat +}$ between subrings of integral elements of a perfectoid Tate ring $R$ and subrings of integral elements of $R^{\flat}$.

Perfectoid fields and perfectoid algebras over them have been introduced by Scholze in his thesis \cite{Scholze}. Kedlaya and Liu \cite{Kedlaya-Liu} defined perfectoid Banach algebras over $\mathbb{Q}_{p}$ that do not necessarily contain a perfectoid field. The general definition of a perfectoid Tate ring stated above has first appeared in Fontaine's Bourbaki talk \cite{Fontaine}. There is also a notion of 'integral perfectoid rings' which we recall and appeal to in the proof of Theorem \ref{Spectrally reduced quotients are perfectoid}~below.  

Let $R$ be a perfectoid algebra over some perfectoid field $K$. The map $R^{\flat} \to R$, $f \mapsto f^{\#}$, described above is certainly far from being surjective (not every element of $R$ has a compatible system of $p$-power roots). However, Scholze has proved a useful approximation lemma (Corollary 6.7 (i) in \cite{Scholze}) which, roughly speaking, states that every element $f \in R$ considered as a function $x \mapsto |f(x)|$ on the adic spectrum $\Spa(R, R^{+})$, with $R^{+}$ a subring of integral elements, can be approximated by a function of the form $g^{\#}$ (for some $g \in R^{\flat}$) at all points where $f$ and $g^{\#}$ are "not too small". From this approximation statement one deduces that, for any ring of integral elements $R^{+}$ of $R$, the natural map $\Spa(R, R^{+}) \to \Spa(R^{\flat}, R^{\flat+})$ given by $x \mapsto x^{\flat}$, where $|g(x^{\flat})| : = |g^{\#}(x)|$, is a homeomorphism (\cite{Scholze}, Corollary 6.7 (iii)). Similarly, the map $\phi \mapsto \phi^{\flat}$ with $\phi^{\flat}(f) = \phi(f^{\#})$ for $f \in R$ defines a homeomorphism between the Berkovich spectrum $\mathcal{M}(R)$ of $R$ and the Berkovich spectrum $\mathcal{M}(R^{\flat})$ of the tilt $R^{\flat}$. To see this, note that the Berkovich spectrum $\mathcal{M}(R)$ can be identified (as a set, not as a topological space) with the set of rank $1$ points of $\Spa(R, R^{\circ})$ (we use here that $R$ is Banach algebra over a nonarchimedean field so that every continuous seminorm on $R$ is bounded, by Lemma \ref{Continuity and boundedness}). The map $\Spa(R, R^{\circ}) \to \Spa(R^{\flat}, R^{\flat \circ}), x \mapsto x^{\flat}$, clearly respects the subsets of rank $1$ points, so its restriction $\phi \mapsto \phi^{\flat}$ is a bijection between $\mathcal{M}(R)$ and $\mathcal{M}(R^{\flat})$. But the map $\phi \mapsto \phi^{\flat}$ is readily seen to be continuous (the topology on the Berkovich spectrum being the topology of pointwise convergence), and a continuous bijection between two compact Hausdorff spaces is a homeomorphism. 

For a general perfectoid Tate ring $R$ which does not necessarily contain a perfectoid field we can fix some power-multiplicative norm defining its topology which again allows us to talk about the Berkovich spectrum of $R$. The result that the map $x \mapsto x^{\flat}$ defines homeomorphisms of adic spectra and of Berkovich spectra holds for any perfectoid Tate ring, as was proved by Kedlaya and Liu in \cite{Kedlaya-Liu}, \cite{Kedlaya-Liu2}. In fact, they first prove the homeomorphism $\mathcal{M}(R) \simeq \mathcal{M}(R^{\flat})$ (\cite{Kedlaya-Liu}, Theorem 3.3.7) by some explicit calculations involving norms on rings of Witt vectors and then deduce the analogous statement for adic spectra from it (see \cite{Kedlaya-Liu}, Theorem 3.6.14, for the case of perfectoid Banach algebras over $\mathbb{Q}_{p}$ and \cite{Kedlaya-Liu2}, Theorem 3.3.16). In their treatment of these questions the role of \cite{Scholze}, Corollary 6.7(iii), is played by \cite{Kedlaya-Liu}, Lemma 3.3.9 (an even more general statement is \cite{Kedlaya-Liu2}, Lemma 3.2.7). Since the language in \cite{Kedlaya-Liu2} is different from that in \cite{Scholze} we explain how a statement analogous to \cite{Scholze}, Corollary 6.7(iii), can be deduced from \cite{Kedlaya-Liu2}, Lemma 3.2.7. 
\begin{lemma} \label{Approximation lemma} Let $R$ be a perfectoid Tate ring with tilt $R^{\flat}$. For every $f \in R$, every $\epsilon > 0$ and every integer $c \geq 0$ there exists an element $g  = \sum_{n=0}^{\infty}p^{n}[\overline{g}_{n}] \in W(R^{\flat\circ})$ such that \begin{equation*} \phi(f - \overline{g}_{0}^{\#}) \leq p^{-1}\max \{\,\phi(\overline{g}_{0}^{\#}), \epsilon \,\} \end{equation*}for all $\phi \in \mathcal{M}(R)$. \end{lemma}
\begin{proof} Up to multiplying $f$ by a power of a pseudo-uniformizer, we may assume that $f \in R^{\circ}$. It follows from the general form of the tilting equivalence (\cite{Kedlaya-Liu2}, Theorem 3.3.8) that the ring of power-bounded elements $R^{\circ}$ of $R$ can be written in the form \begin{equation*} R^{\circ} = W(R^{\flat\circ})/zW(R^{\flat\circ}) \end{equation*}where $z = \sum_{n=0}^{\infty}p^{n}[\overline{z}_{n}] \in W(R^{\flat\circ})$ is a primitive element of degree $1$, which means here that $z - [\overline{z}_{0}] \in pW(R^{\flat\circ})^{\times}$ (see \cite{Kedlaya-Liu}, Definition 3.3.4). 

By \cite{Kedlaya-Liu}, Theorem 3.3.7, a bounded multiplicative seminorm $\phi$ on $R$ can be recovered from its image $\phi^{\flat}$ in $\mathcal{M}(R^{\flat})$ as follows. We first lift the restriction of $\phi^{\flat}$ to $R^{\flat\circ}$ to a seminorm $\lambda(\phi^{\flat})$ on $W(R^{\flat\circ})$ defined by \begin{equation*} \lambda(\phi^{\flat})(\sum_{n=0}^{\infty}p^{n}[\overline{f}_{n}]) = \sup_{n} \{p^{-n}\phi^{\flat}(\overline{f}_{n})\} \end{equation*}which is a multiplicative seminorm bounded by the $p$-adic norm, by \cite{Kedlaya13}, Lemma 4.1. Then \cite{Kedlaya-Liu}, Theorem 3.3.7, tells us that the restriction of $\phi$ to $R^{\circ} = W(R^{\flat\circ})/zW(R^{\flat\circ})$ is precisely the quotient seminorm induced by the seminorm $\lambda(\phi^{\flat})$ on $W(R^{\flat\circ})$. Now, given $\epsilon > 0$ and an integer $c \geq 0$, \cite{Kedlaya-Liu}, Lemma 3.3.9, allows us to lift our element $f \in R^{\circ}$ (along the quotient map $W(R^{\flat\circ}) \to W(R^{\flat\circ})/zW(R^{\flat\circ}))$ to an element $g = \sum_{n=0}^{\infty}p^{n}[\overline{g}_{n}] \in W(R^{\flat\circ})$ such that \begin{equation} \alpha(\overline{g}_{1}) \leq \max \{p^{-p^{-1} - \dots - p^{-c}}\alpha(\overline{g}_{0}), \epsilon \} \end{equation}and \begin{equation} \alpha(\overline{g}_{n}) \leq \max \{\alpha(\overline{g}_{0}), \epsilon \} \ \ \ \text{for} \ n>1\end{equation}for all $\alpha \in \mathcal{M}(R^{\flat})$.  

The map $R^{\flat\circ} \to R^{\circ}, f \mapsto f^{\#}$, coincides with the composite map $R^{\flat\circ} \to W(R^{\flat\circ}) \to W(R^{\flat\circ})/zW(R^{\flat\circ})$ where the first arrow is the usual Teichmüller map $[\cdot]: R^{\flat\circ} \to W(R^{\flat\circ})$ and the second arrow is reduction modulo $z$. Hence, for every $\phi \in \mathcal{M}(R)$, we have \begin{equation} \phi(f - \overline{g}^{\#}_{0}) \leq \lambda(\phi^{\flat})(g - [\overline{g}_{0}]) = \lambda(\phi^{\flat})(\sum_{n=1}^{\infty}p^{n}[\overline{g}_{n}]) = \sup_{n>0}\{p^{-n}\phi^{\flat}(\overline{g}_{n})\}. \end{equation}Given $\phi \in \mathcal{M}(R)$, let $m>0$ be an integer such that $p^{-m}\phi^{\flat}(\overline{g}_{m}) = \max_{n>0}\{p^{-n}\phi^{\flat}(\overline{g}_{n})\}$. If $m=1$, then by equations (1) and (3) we have \begin{equation*} \phi(f-\overline{g}^{\#}_{0}) \leq p^{-1}\max\{p^{-p^{-1} - \dots - p^{-c}}\phi^{\flat}(\overline{g}_{0}), \epsilon \} = p^{-1}\max\{p^{-p^{-1} - \dots - p^{-c}}\phi(\overline{g}^{\#}_{0}), \epsilon \}. \end{equation*}On the other hand, if $m>1$, then equations (2) and (3) imply \begin{equation*}\phi(f-\overline{g}^{\#}_{0}) \leq p^{-m}\max \{\phi^{\flat}(\overline{g}_{0}), \epsilon \} = p^{-m}\max \{\phi(\overline{g}^{\#}_{0}), \epsilon \}< p^{-1}\max \{\phi(\overline{g}^{\#}_{0}), \epsilon \}.\end{equation*}

Now, if $\epsilon$ belongs to the open interval $(p^{-p^{-1}-\dots-p^{-c}}\phi(\overline{g}^{\#}_{0}), \phi(\overline{g}^{\#}_{0}))$, then in the case $m=1$ we obtain $\phi(f-\overline{g}^{\#}_{0}) \leq p^{-1}\epsilon < p^{-1}\phi(\overline{g}^{\#}_{0})$ and in the case $m>1$ we obtain $\phi(f-\overline{g}^{\#}_{0})<p^{-1}\phi(\overline{g}^{\#}_{0})$. If $\epsilon \geq \phi(\overline{g}^{\#}_{0})$, then a fortiori $\epsilon > p^{-p^{-1}-\dots-p^{-c}}\phi(\overline{g}^{\#}_{0})$ and hence $\phi(f-\overline{g}^{\#}_{0}) \leq p^{-1}\epsilon$ both in the case $m=1$ and in the case $m>1$. Finally, if $\epsilon < p^{-p^{-1}-\dots-p^{-c}}\phi(\overline{g}^{\#}_{0})$, then also $\epsilon < \phi(\overline{g}^{\#}_{0})$ and hence, for any $m \geq 1$, we have $\phi(f-\overline{g}^{\#}_{0})<p^{-1}\phi(\overline{g}^{\#}_{0})$. This ends the proof of the lemma.\end{proof}We view Lemma \ref{Approximation lemma} as an analogue of \cite{Scholze}, Corollary 6.7(iii), in the context of Berkovich spectra of general perfectoid Tate rings. As in the case of loc.~cit., there is the following simple consequence which is often all that is needed for applications.\begin{cor} \label{Approximation lemma 2} Consider a perfectoid Tate ring $R$ with tilt $R^{\flat}$. For every $\epsilon > 0$ and every $f \in R$ there exists an element $g \in R^{\flat}$ such that for every $\phi \in \mathcal{M}(R)$ either $\phi(f) = \phi(g^{\#})$ or $\max \{\phi(f), \phi(g^{\#})\} < \epsilon$.\end{cor}
\begin{proof} We let $g$ be the element $\overline{g}_{0} \in R^{\flat}$ from Lemma \ref{Approximation lemma}. Then for every $\phi \in \mathcal{M}(R)$ we have $\phi(f-g^{\#}) \leq p^{-1}\max\{\phi(g^{\#}), \epsilon\}$. Suppose that $\max\{\phi(f), \phi(g^{\#})\} \geq \epsilon$. If $\phi(g^{\#}) \geq \phi(f)$, then $\phi(g^{\#}) \geq \epsilon$. Then \begin{equation*} \phi(f-g^{\#}) \leq p^{-1}\phi(g^{\#})<\phi(g^{\#}) \end{equation*}and therefore \begin{equation*}\phi(f) = \phi((f-g^{\#})+g^{\#}) = \phi(g^{\#}).\end{equation*}On the other hand, if $\phi(f)>\phi(g^{\#})$, then $\phi(f-g^{\#}) = \phi(f)$. If $phi(g^{\#}) \geq \epsilon$, then $\phi(f) = \phi(f-g^{\#})<\phi(g^{\#})$, a contradiction. It follows that $\phi(g^{\#})<\epsilon$ and then $\phi(f) = \phi(f-g^{\#})\leq p^{-1}\epsilon<\epsilon$ which contradicts the assumption that $\max\{\phi(f), \phi(g^{\#})\}\geq \epsilon$. This shows that for any $\phi \in \mathcal{M}(R)$ the condition $\max\{\phi(f), \phi(g^{\#})\}\geq \epsilon$ implies $\phi(f) = \phi(g^{\#})$, as desired.\end{proof}
This approximation statement allows us to prove the following proposition which compares the Shilov boundaries of a perfectoid Tate ring (containing a nonarchimedean field) and of its tilt. In the proof we identify the Berkovich spectrum (as a set) with the set of rank $1$ points in the adic spectrum of $(R, R^{\circ})$ and denote its elements by letters $x, y$ instead of the $\phi, \alpha$ used in the rest of the paper, to make the connection with \cite{Scholze}, Corollary 6.7(iii), more transparent.
\begin{prop} \label{Shilov boundary and tilting} Let $K$ be a nonarchimedean field and let $R$ be a perfectoid $K$-algebra with tilt $R^{\flat}$. The homeomorphism of Berkovich spectra $\mathcal{M}(R) \simeq \mathcal{M}(R^{\flat})$ identifies the Shilov boundaries of $R$ and $R^{\flat}$.\end{prop}
\begin{proof} Let $\Sigma$ be any boundary of $\mathcal{M}(R)$. Let $g \in R^{\flat}$. Since $\Sigma$ is a boundary there exists $x \in \Sigma$ such that $|g(x^{\flat})| = |g^{\#}(x)| = \lVert g^{\#} \rVert = \lVert g \rVert$. This shows that the image $\Sigma^{\flat}$ of $\Sigma$ is a boundary of $\mathcal{M}(R^{\flat})$. We claim that, conversely, every boundary of $\mathcal{M}(R^{\flat})$ is of the form $\Sigma^{\flat}$ for some boundary $\Sigma$ of $\mathcal{M}(R)$. Since the map $x \mapsto x^{\flat}$ is bijective it suffices to show: If $\Sigma \subset \mathcal{M}(R)$ is a subset such that $\Sigma^{\flat}$ is a boundary of $\mathcal{M}(R^{\flat})$, then $\Sigma$ is a boundary of $\mathcal{M}(R)$. To see this, suppose that $\Sigma \subset \mathcal{M}(R)$ is not a boundary. Then there exist $f \in R$ such that $\Vert f \rVert > |f(x)|$ for all $x \in \Sigma$. Since $\mathcal{M}(R)$ is compact and the norm on $R$ is power-multiplicative (and hence coincides with its spectral seminorm), there exists $y \in \mathcal{M}(R)$ with $|f(y)| = \lVert f \rVert$, thus $|f(y)| > |f(x)|$ for all $x \in \Sigma$. Fix an element $x \in \Sigma$. If $|f(x)|>0$, choose $\epsilon > 0$ with $|f(x)| \geq \epsilon$ (whence we automatically have $|f(y)| > \epsilon$) and if $|f(x)|=0$ choose $\epsilon>0$ with $|f(y)| \geq \epsilon$. For this choice of $\epsilon>0$, pick an element $g \in R^{\flat}$ as in Corollary \ref{Approximation lemma 2}. Then the corollary implies $|f(x)| = |g^{\#}(x)|$ and $|f(y)| = |g^{\#}(y)|$ if $|f(x)|>0$. If $|f(x)|=0$, the corollary implies $|g^{\#}(x)|<\epsilon$ and $|f(y)| = |g^{\#}(y)|$. It follows that \begin{equation*}|g(x^{\flat})| = |g^{\#}(x)| = |f(x)| < |f(y)| = |g^{\#}(y)| = |g(y^{\flat})| \end{equation*} if $|f(x)|>0$ and \begin{equation*} |g(x^{\flat})| = |g^{\#}(x)| < \epsilon \leq |f(y)| = |g^{\#}(y)| = |g(y^{\flat})| \end{equation*}if $|f(x)|=0$. Since $x \in \Sigma$ was arbitrary we conclude that $\Sigma^{\flat}$ is not a boundary of $\mathcal{M}(R^{\flat})$, proving the claim. \end{proof} 
Recall from Section 2 that an ideal $I$ of a seminormed ring $R$ is said to be spectrally reduced if there exists some bounded power-multiplicative seminorm on $R$ whose kernel is $I$ (equivalently, $I$ is spectrally reduced if the spectral seminorm derived from the quotient norm on $R/I$ is a norm on $R/I$, see Lemma \ref{Properties of spectrally reduced ideals}). Given a uniform complete Tate ring $R$ with pseudo-uniformizer $\varpi \in R$, Remark 2.8.18 in \cite{Kedlaya-Liu} (in conjunction with Remark \ref{Uniform units}) allows us to choose a power-multiplicative norm $\lVert\cdot\rVert$ defining the topology on $R$ such that $\varpi$ becomes multiplicative for the norm $\lVert\cdot\rVert$. We say that an ideal $I$ of the uniform Tate ring $R$ is spectrally reduced if it is spectrally reduced as an ideal of the normed ring $(R, \lVert\cdot\rVert)$. By Remark \ref{Topological spectrum might depend on the seminorm} the property of an ideal $I \subsetneq R$ being spectrally reduced depends only on the topology of the uniform Tate ring $R$ and does not depend on the choice of the norm $\lVert\cdot\rVert$.   

We are interested in describing quotients of a perfectoid ring $R$ modulo its spectrally reduced ideals. When $R$ is of characteristic $p$, the quotient $R/I$ of $R$ modulo any closed radical ideal $I$ is perfect, hence also uniform, by an application of Banach's Open Mapping Theorem (view $R/I$ as a Banach module over itself via the map $x \mapsto x^{p}$).~Thus in characteristic $p$ every such quotient is already perfectoid. One can ask if an analogue of this statement holds true in the characteristic $0$ case: If $R$ is a perfectoid ring of zero characteristic and $I$ is a spectrally reduced ideal of $R$, is the quotient $R/I$ perfectoid? An affirmative answer to this question follows from a recent result of Bhatt and Scholze (\cite{Prismatic}, Theorem 7.4).
\begin{thm} \label{Spectrally reduced quotients are perfectoid} Let $R$ be a perfectoid Tate ring and let $I \subsetneq R$ be an ideal of $R$. Then the quotient $R/I$ is perfectoid if and only if the ideal $I$ is spectrally reduced.\end{thm}
The above theorem is a consequence of \cite{Prismatic}, Theorem 7.4, as has already been observed in Theorem 2.9.12 of the final version of Kedlaya's lecture notes \cite{Kedlaya17}. For the convenience of the reader we show how the assertion of the theorem can be deduced from the relevant assertions in \cite{Prismatic} (see also Remark 7.5 of that work for related discussion). Recall that the uniformization $R^{u}$ of a Banach ring $R$ is the completion of $R$ with respect to its spectral seminorm. The following lemma records a simple property of the uniformization. 
\begin{lemma} \label{Universal property of uniformization} Let $R$ be a Banach ring and let $R^{u}$ be the uniformization of $R$. Then $R^{u}$ satisfies the following universal property: Every bounded homomorphism from $R$ to a uniform Banach ring $S$ factors uniquely through the canonical map $u: R \to R^{u}$. 
\end{lemma} In other words, the functor $R \mapsto R^{u}$ from the category of Banach rings to the category of uniform Banach rings is left-adjoint to the natural inclusion functor. \begin{proof}[Proof of the lemma] Let $\varphi: R \to S$ be a bounded homomorphism from $R$ to a uniform Banach ring $S$. By Remark \ref{The spectral seminorm for non-complete rings} the spectral seminorm on any seminormed ring is its maximal power-multiplicative bounded seminorm. It follows that $\varphi$ is bounded if $R$ is endowed with its spectral seminorm and $S$ is endowed with a power-multiplicative norm $\lVert \cdot \rVert$ defining its topology. Let $f \in R^{u}$ be an arbitrary element of the uniformization. Then there exists a sequence $(f_n)_{n} \subset R$ such that $u(f_n) \to f$ in $R^{u}$. The map $u: R \to R^{u}$ is an isometry with respect to the spectral seminorm $|\cdot|_{\spc}$ on $R$. Now set $\varphi^{u}(f) := \lim_{n}\varphi(f_{n})$. If $(g_{n})_{n} \subset R$ is another sequence with the property that $u(g_{n}) \to f$, then $|g_{n} - f_{n}|_{\spc} \to 0$ as $n \to \infty$ and thus \begin{equation*} \lVert \varphi(f_{n}) - \varphi(g_{n}) \rVert \leq |f_{n} - g_{n}|_{\spc} \to 0. \end{equation*}This shows that $\lim_{n}\varphi(g_{n}) = \lim_{n}\varphi(f_{n})$ in $S$ and hence $\varphi^{u}$ is a well-defined map $R^{u} \to S$ with the property that $\varphi^{u} \circ u = \varphi$. 

We now verify that $\varphi^{u}$ is a bounded homomorphism. Let $(f_{n})_{n}, (g_{n})_{n}$ be sequences in $R$ with $\lim_{n}u(f_n) = f, \lim_{n}u(g_n) = g$. Then $\lim_{n}u(f_n + g_n) = f + g$ and $\lim_{n}u(f_{n}g_{n}) = fg$, the map $u: R \to R^{u}$ being a homomorphism. Since $\varphi$ is a homomorphism we obtain $\varphi^{u}(f + g) = \lim_{n}\varphi(u(f_{n} + g_{n})) = \lim_{n}\varphi(u(f_{n})) + \lim_{n}\varphi(u(g_{n})) = \varphi^{u}(f) + \varphi^{u}(g)$ and, similarly, $\varphi^{u}(fg) = \lim_{n}\varphi(u(f_{n}g_{n})) = \lim_{n}\varphi(u(f_{n})) \varphi(u(g_{n})) = \varphi^{u}(f)\varphi^{u}(g)$. This shows that $\varphi^{u}$ is a homomorphism. By definition, this homomorphism is bounded on the image of $u$, which is a dense subring of $R^{u}$. Hence it must be bounded on all of $R^{u}$. 

It remains to prove uniqueness of the homomorphism $\varphi^{u}$. For this, let $\psi$ be another bounded homomorphism $R^{u} \to S$ with $\psi \circ u = \phi$. For $f \in R^{u}$, let $(f_{n})_{n}$ be any sequence with $u(f_{n}) \to f$. Then continuity of the maps $\psi, \phi^{u}$ and $\phi$ implies $\psi(f) = \psi(\lim_{n}u(f_{n})) = \lim_{n}\psi(u(f_{n})) = \lim_{n}\phi(f_{n}) = \lim_{n}\phi^{u}(u(f_{n})) = \phi^{u}(\lim_{n}u(f_{n})) = \phi^{u}(f)$, showing that $\psi = \phi^{u}$.\end{proof}
\begin{proof}[Proof of Theorem \ref{Spectrally reduced quotients are perfectoid}] Suppose that the quotient $R/I$ is perfectoid. Then $R/I$ is uniform and, in particular, $R/I$ is a spectrally reduced normed ring. By Lemma \ref{Properties of spectrally reduced ideals}(2) this means that the ideal $I$ is spectrally reduced. 

We now prove that the quotient of a perfectoid Tate ring $R$ by a spectrally reduced ideal $I$ is perfectoid. By case (iii) of \cite{Kedlaya-Liu2}, Theorem 3.3.18, the quotient $R/I$ is perfectoid if and only if it is uniform. So, we have to show that $R/I$ is a uniform Banach ring for every spectrally reduced ideal $I$ of $R$. 

If $R_{0}$ is a ring which is $\varpi$-adically complete for some non-zero divisor $\varpi$ such that $\varpi^{p}$ divides $p$, then we say that $R_{0}$ is an integral perfectoid ring if $R_{0}[\varpi^{-1}]$, endowed with the $\varpi$-adic topology on $R_{0}$, is a perfectoid Tate ring. Usually integral perfectoid rings are defined in a slightly different way (see \cite{BMS}, Definition 3.5) but Lemma 3.20 and Lemma 3.21 in \cite{BMS} show that for rings which are $\varpi$-adically complete for some non-zero-divisor $\varpi$ the two definitions are equivalent.

Fix a topologically nilpotent unit $\varpi \in R$ as in the definition of a perfectoid Tate ring. Then the topology on the subring $R^{\circ}$ is the $\varpi$-adic topology. In particular, $R^{\circ}$ is $\varpi$-adically complete. Let $I$ be a spectrally reduced ideal of $R$. Let $S = R^{\circ}/(I \cap R^{\circ})$. By Lemma \ref{Properties of spectrally reduced ideals} we know that $I$ is closed in $R$, so the ideal $I\cap R^{\circ}$ of $R^{\circ}$ is closed with respect to the $\varpi$-adic topology on $R^{\circ}$. Recall that the topologically nilpotent unit $\varpi$ is such that $\varpi^{p}$ divides $p$ in $R^{\circ}$. Thus $(p)_{R^{\circ}} \subseteq (\varpi^{p})_{R^{\circ}} \subseteq (\varpi)_{R^{\circ}}$. By \cite{Stacks}, Part 1, Lemma 10.96.8, this entails that $R^{\circ}$ is $p$-adically complete. The inclusion $(p)_{R^{\circ}} \subseteq (\varpi)_{R^{\circ}}$ tells us that the $p$-adic topology on $R^{\circ}$ is finer than the $\varpi$-adic topology. It follows that the ideal $I \cap R^{\circ}$ is also $p$-adically closed in $R^{\circ}$. Consequently, the quotient $S = R^{\circ}/(I \cap R^{\circ})$ is $p$-adically complete (and, a fortiori, derived $p$-adically complete). Thus $S$ is a semiperfectoid ring in the sense of Bhatt-Scholze \cite{Prismatic}, i.e.~$S$ is a derived $p$-adically complete quotient of an integral perfectoid ring. 

By \cite{Prismatic}, Corollary 7.3, there exists a universal integral perfectoid ring $S_{\perfd}$ equipped with a map $S \to S_{\perfd}$. Then $R' = S_{\perfd}[\varpi^{-1}]$ is a perfectoid Tate ring equipped with a continuous ring map $R \to R'$ such that $I$ is in the kernel of $R \to R'$ and such that every continuous ring map $\varphi: R \to R''$ from $R$ to a perfectoid Tate ring $R''$ containing $I$ in its kernel factors through $R \to R'$. On the other hand, the uniformization $(R/I)^{u}$ of $R/I$ is a perfectoid Tate ring by virtue of \cite{Kedlaya-Liu2}, Theorem 3.3.18 (ii), and, by Lemma \ref{Universal property of uniformization}, $(R/I)^{u}$ satisfies the same universal property in the category of perfectoid $R$-algebras (and even in the category of uniform Banach $R$-algebras) as the universal property just described for $R'$. It follows that $R' = (R/I)^{u}$. But by \cite{Prismatic}, Theorem 7.4, the map $S \to S_{\perfd}$ is surjective, so the map $R/I \to (R/I)^{u}$ is surjective. Since $I$ was assumed to be spectrally reduced, the map $R/I \to (R/I)^{u}$ is also injective. This shows that $R/I$ is already a uniform Banach ring, as claimed.
\end{proof}
\begin{cor} \label{Maximal ideals in perfectoid rings} Let $R$ be a perfectoid Tate ring and $\mathfrak{m} \subsetneq R$ a maximal ideal. Then the quotient Banach field $R/\mathfrak{m}$ is a perfectoid field.
\end{cor}
\begin{proof} This is \cite{Kedlaya17}, Corollary 2.9.14. It follows from the above theorem by combining it with \cite{Kedlaya18}, Theorem 4.2.
\end{proof}
\begin{cor} \label{Perfectoid algebras over an algebraically closed field are topologically Jacobson} Every perfectoid algebra $R$ over an algebraically closed nonarchimedean field $K$ is topologically Jacobson.
\end{cor}
\begin{proof} Let $I$ be a spectrally reduced prime ideal of $R$. We want to show that $I$ is an intersection of maximal ideals. By Theorem \ref{Spectrally reduced quotients are perfectoid} the quotient Banach algebra $R/I$ is uniform. But by \cite{EM2}, Corollary a), every uniform Banach algebra over an algebraically closed nonarchimedean field is Jacobson-semisimple. This proves that $I$ is an intersection of maximal ideals of $R$, as desired.   
\end{proof}
Let $R$ be a perfectoid Tate ring with tilt $R^{\flat}$. For $I \subsetneq R$ a closed ideal define a subset $I^{\flat}$ of $R^{\flat}$ by 
\begin{equation*} I^{\flat} = \{\, (f^{(n)})_{n \in \mathbb{N}} \in R^{\flat} \mid f^{(n)} \in I \ \text{for all}~n \, \}. \end{equation*} Since $I$ is an ideal, if $f, g \in R^{\flat}$ with $g \in I^{\flat}$, then $fg \in I^{\flat}$. Since $I$ is moreover a closed ideal, for $f, g \in I^{\flat}$, the limit $(f + g)^{(n)} = \lim_{m \to \infty}(f^{(m + n)} + g^{(m + n)})^{p^{m}}$ is in $I$ for every $n$ so that $f + g \in I^{\flat}$. We also see that $I^{\flat}$ is not the whole ring $R^{\flat}$ since otherwise $I$ would contain the element $1 = 1^{(0)}$. This shows that the subset $I^{\flat}$ is a proper ideal of $R^{\flat}$. We call it the \textit{tilt} of the closed ideal $I$ of $R$. Note that $I^{\flat} = \{\, (f^{(n)})_{n \in \mathbb{N}} \in R^{\flat} \mid f^{(0)} \in I \, \}$ whenever $I$ is a closed radical ideal. 

When the ideal $I$ is spectrally reduced, the ideal $I^{\flat}$ has a particularly simple description as we show next.
\begin{lemma} \label{The tilt of an ideal and perfectoid quotients} Let $R$ be a perfectoid Tate ring and let $I \subsetneq R$ be a spectrally reduced ideal so that the quotient Banach ring $R/I$ is again perfectoid. Let $\varphi: R \to R/I$ denote the quotient map and $\varphi^{\flat}$ its tilt. Then $\ker({\varphi}^{\flat}) = I^{\flat}$, and $(R/I)^{\flat}$ is topologically isomorphic to $R^{\flat}/I^{\flat}$. In particular, the ideal $I^{\flat}$ is closed.
\end{lemma}
\begin{proof} Recall that the map ${\varphi}^{\flat}: R^{\flat} \to (R/I)^{\flat}$ is given by \begin{equation*} {\varphi}^{\flat}(f) = (\varphi(f^{(n)}))_{n}, \hspace{5mm} f = (f^{(n)})_{n \in \mathbb{N}_{0}} \in R^{\flat}. \end{equation*} This shows that $\ker({\varphi}^{\flat}) = \ker(\varphi)^{\flat} = I^{\flat}$. In particular, $I^{\flat}$ is closed in $R^{\flat}$ and $R^{\flat}/I^{\flat}$ is a Banach ring. By case (iii) of \cite{Kedlaya-Liu2}, Theorem 3.3.18(b), the bounded homomorphism $\varphi^{\flat}$ is surjective and thus also strict by Banach's Open Mapping Theorem. Hence the Banach rings $R^{\flat}/I^{\flat}$ and $(R/I)^{\flat}$ are isomorphic.
\end{proof}
We obtain a well-defined map $I \mapsto I^{\flat}$ from the set of spectrally reduced ideals of $R$ to the set of spectrally reduced ideals of $R^{\flat}$. We want to define an inverse to this map. Note that any spectrally reduced ideal in $R$ resp.~in $R^{\flat}$ is an intersection of ideals of the form $\ker(\phi)$ for $\phi \in \mathcal{M}(R)$ resp.~for $\phi \in \mathcal{M}(R^{\flat})$ (this is a consequence of \cite{Berkovich}, Theorem 1.3.1). Let $\phi \mapsto \phi^{\#}$ denote the inverse map of the homeomorphism $\mathcal{M}(R) \simeq \mathcal{M}(R^{\flat})$. For $J \subsetneq R^{\flat}$ a spectrally reduced ideal (equivalently, a closed radical ideal) define \begin{equation*} J^{\#} = \bigcap_{\substack{\phi \in \mathcal{M}(R^{\flat}) \\ J \subseteq \ker(\phi)}}\ker(\phi^{\#}). \end{equation*} In particular, if $J = \ker(\phi)$ for some $\phi \in \mathcal{M}(R^{\flat})$, then $J^{\#} = \ker(\phi^{\#})$. It is readily seen that the map $J \mapsto J^{\#}$ thus defined respects inclusions and intersections (for every $\phi \in \mathcal{M}(R^{\flat})$ the kernel $\ker(\phi)$ is a prime ideal). Then, for every spectrally reduced ideal $J$ of $R^{\flat}$, we have   
\begin{equation*}
(J^{\#})^{\flat} = \bigcap_{\substack{\phi \in \mathcal{M}(R^{\flat}) \\ J \subseteq \ker(\phi)}} \ker(\phi^{\#})^{\flat} = \bigcap_{\substack{\phi \in \mathcal{M}(R^{\flat}) \\ J \subseteq \ker(\phi)}} \ker(\phi^{\# \flat}) = \bigcap_{\substack{\phi \in \mathcal{M}(R^{\flat}) \\ J \subseteq \ker(\phi)}} \ker(\phi) = J.
\end{equation*} Conversely, let $I$ be a spectrally reduced ideal of $R$. The quotient $R/I$ is a perfectoid ring with tilt $R^{\flat}/I^{\flat}$ (Theorem \ref{Spectrally reduced quotients are perfectoid} and Lemma \ref{The tilt of an ideal and perfectoid quotients}). Thus we have a homeomorphism $\mathcal{M}(R/I) \simeq \mathcal{M}(R^{\flat}/I^{\flat})$, which, together with Lemma \ref{Berkovich spectrum of a quotient}, shows that for $\phi \in \mathcal{M}(R^{\flat})$ the condition $\ker(\phi) \supseteq I^{\flat}$ is equivalent to $\ker(\phi^{\#}) \supseteq I$. It follows that \begin{equation*} 
(I^{\flat})^{\#} = (\bigcap_{\substack{\phi \in \mathcal{M}(R^{\flat}) \\ \ker(\phi) \supseteq I^{\flat} }} \ker(\phi))^{\#} = \bigcap_{\substack{\phi \in \mathcal{M}(R^{\flat}) \\ \ker(\phi) \supseteq I^{\flat} }} \ker(\phi)^{\#} = \bigcap_{\substack{\phi \in \mathcal{M}(R) \\ \ker(\phi) \supseteq I }} \ker(\phi) = I. 
\end{equation*} We summarize our observations in the following proposition.\begin{prop} \label{Spectrally reduced ideals and tilting} Let $R$ be a perfectoid Tate ring. There is an inclusion-preserving bijection between the set of spectrally reduced ideals of $R$ and the set of spectrally reduced ideals of $R^{\flat}$ given by the maps \begin{equation*} I \mapsto I^{\flat} = \{\, (f^{(n)})_{n \in \mathbb{N}} \in R^{\flat} \mid f^{(0)} \in I \, \} \end{equation*} and \begin{equation*} J \mapsto J^{\#} = \bigcap_{\substack{\phi \in \mathcal{M}(R^{\flat}) \\ J \subseteq \ker(\phi)}}\ker(\phi^{\#}). \end{equation*} \end{prop}
We note some straightforward consequences of Proposition \ref{Spectrally reduced ideals and tilting}.
\begin{cor} \label{Kernels of multiplicative seminorms and tilting} Let $R$ be a perfectoid Tate ring with tilt $R^{\flat}$ and suppose that $\phi_{1}, \phi_{2} \in \mathcal{M}(R)$ are two bounded multiplicative seminorms. Then $\ker(\phi_{1}^{\flat}) = \ker(\phi_{2}^{\flat})$ if and only if $\ker(\phi_{1}) = \ker(\phi_{2})$. \end{cor}
\begin{proof} For any $\phi \in \mathcal{M}(R)$ we have $\ker(\phi^{\flat}) = \ker(\phi)^{\flat}$, so the corollary follows from the part of Proposition \ref{Spectrally reduced ideals and tilting} asserting injectivity of the map $I \mapsto I^{\flat}$.\end{proof}
\begin{cor} \label{Absolute values and tilting} Let $R$ be a perfectoid Tate ring with tilt $R^{\flat}$. A bounded multiplicative seminorm $\phi \in \mathcal{M}(R)$ is an absolute value if and only if $\phi^{\flat} \in \mathcal{M}(R^{\flat})$ is an absolute value.\end{cor}
\begin{proof} The zero ideal in any uniform normed ring is spectrally reduced and the tilt of the zero ideal of $R$ is given by the zero ideal of $R^{\flat}$. Hence the corollary again follows from the injectivity part of Proposition \ref{Spectrally reduced ideals and tilting}.
\end{proof}
\begin{cor} \label{Topological zero-divisors and tilting} Let $R$ be a perfectoid Tate ring which contains some nonarchimedean field. Then $R$ has no non-zero topological divisors of zero if and only if its tilt $R^{\flat}$ has no non-zero topological divisors of zero.
\end{cor}
\begin{proof} Combine Corollary \ref{Absolute values and tilting}, Proposition \ref{Shilov boundary and tilting} and Escassut's theorem on topological divisors of zero (Theorem \ref{Escassut's theorem}).  
\end{proof} There are also analogues of Corollary \ref{Kernels of multiplicative seminorms and tilting} and Corollary \ref{Absolute values and tilting} for elements in the adic spectrum $\Spa(R, R^{+})$, with $R^{+}$ a ring of integral elements of $R$. Before we show this let us recall from \cite{Wedhorn}, Remark and Definition 4.12, that an element $z$ in the adic spectrum $\Spa(R, R^{+})$ of a Huber pair $(R, R^{+})$ is called a secondary (or vertical) generization of an element $x \in \Spa(R, R^{+})$ if $z$ is a generization of $x$ (that is, $x$ is contained in the closure of the singleton set $\{z\}$ in $\Spa(R, R^{+})$) and $\ker(x) = \ker(z)$ (in fact, the reference \cite{Wedhorn} talks about generizations in the valuation spectrum of $R$ which contains $\Spa(R, R^{+})$ as a subspace). The original definition of secondary generizations, given by Huber in \cite{Huber1}, \S 2, was phrased in somewhat different language; equivalence of the two definitions follows, for example, from Proposition 2.14 in \cite{Wedhorn}. In the following we identify an equivalence class of continuous valuations $x \in \Spa(R, R^{+})$ with a representative of that equivalence class and thus we view elements of $\Spa(R, R^{+})$ as continuous valuations on $R$ which are $\leq 1$ on the subring $R^{+}$. Note that the ideal of $R$ which we call the kernel of a valuation $x$ on $R$ and denote $\ker(x)$ is usually called the support of $x$ and denoted $\supp(x)$ in the literature on adic spaces.
\begin{lemma} \label{Kernels of valuations are spectrally reduced} Let $R$ be a perfectoid Tate ring and $R^{+}$ a ring of integral elements of $R$. Every prime ideal of $R$ which is of the form $\ker(x)$ for some $x \in \Spa(R, R^{+})$ is spectrally reduced.\end{lemma}
\begin{proof} Let $x \in \Spa(R, R^{+})$. Since $R$ is Tate, the adic spectrum $\Spa(R, R^{+})$ is an analytic adic space, so all generizations of $x$ are secondary by \cite{Huber1}, \S 3. Moreover, by \cite{Huber3}, Lemma 1.1.10 (i) and (ii), the set of all generizations of $x$ is totally ordered and the maximal element in this totally ordered set is a unique rank $1$ generization $z$ of $x$. In particular, $z$ is a continuous multiplicative seminorm on $R$ such that $\ker(x) = \ker(z)$. But by Remark \ref{Topological spectrum might depend on the seminorm} we know that the kernel of any continuous multiplicative seminorm on a uniform complete Tate ring $R$ is spectrally reduced. It follows that $\ker(x) = \ker(z)$ is spectrally reduced, as claimed.\end{proof}
\begin{cor} \label{Kernels of valuations and tilting} Let $R$ be a perfectoid Tate ring and let $R^{+} \subseteq R^{\circ}$ be a subring of integral elements of $R$. If $x, y \in \Spa(R, R^{+})$ are continuous valuations such that $\ker(x^{\flat}) = \ker(y^{\flat})$, then $\ker(x) = \ker(y)$. \end{cor}
\begin{proof} We have $\ker(x^{\flat}) = \ker(x)^{\flat}$ for every $x \in \Spa(R, R^{+})$. By Lemma \ref{Kernels of valuations are spectrally reduced} the ideals $\ker(x)$ and $\ker(y)$ are spectrally reduced, whence the corollary follows by the injectivity part of Proposition \ref{Spectrally reduced ideals and tilting}. \end{proof}\begin{cor} \label{Absolute values of higher rank and tilting} Let $R$ be a perfectoid Tate ring, let $R^{+}$ be a ring of integral elements of $R$ and let $x \in \Spa(R, R^{+})$. Then $\ker(x) = 0$ if and only if $\ker(x^{\flat}) = 0$.\end{cor}
\begin{proof} By Lemma \ref{Kernels of valuations are spectrally reduced}, the ideal $\ker(x)$ is spectrally reduced. Since the zero ideal is spectrally reduced in any uniform normed ring, we conclude by Proposition \ref{Spectrally reduced ideals and tilting}. \end{proof} As in the rest of the paper we write $\TopSpec(R)$ for the topological spectrum of a seminormed ring $R$, i.e.~the subspace of $\Spec(R)$ which consists of spectrally reduced prime ideals. As we show next, the bijection $I \mapsto I^{\flat}$ introduced above restricts to a homeomorphism $\TopSpec(R) \simeq \TopSpec(R^{\flat})$. 
\begin{thm} \label{Prime ideals and tilting} Let $R$ be a perfectoid Tate ring with tilt $R^{\flat}$. For every spectrally reduced prime ideal $\mathfrak{p}$ the subset \begin{equation*} \mathfrak{p}^{\flat} = \{\, (f^{(n)})_{n \in \mathbb{N}_{0}} \in R^{\flat} \mid f^{(0)} \in \mathfrak{p} \,\} \end{equation*} is a spectrally reduced prime ideal of $R^{\flat}$. The map \begin{equation*} \TopSpec(R) \to \TopSpec(R^{\flat}), \mathfrak{p} \mapsto \mathfrak{p}^{\flat}, \end{equation*} is a homeomorphism for the Zariski topology on $\TopSpec(R)$ and $\TopSpec(R^{\flat})$, with inverse given by \begin{equation*} \mathfrak{q} \mapsto \mathfrak{q}^{\#} = \bigcap_{\substack{\phi \in \mathcal{M}(R^{\flat}) \\ \ker(\phi) \supseteq \mathfrak{q}}}\ker(\phi^{\#}). \end{equation*} \end{thm}
\begin{cor} \label{Integral domains and tilting} Let $R$ be a perfectoid Tate ring with tilt $R^{\flat}$. Then $R$ is an integral domain if and only if $R^{\flat}$ is an integral domain. 
\end{cor} \begin{proof}[Proof of the corollary] The zero ideal of $R$ is spectrally reduced and tilts to the zero ideal of $R^{\flat}$. Hence the theorem implies that the zero ideal of $R$ is a prime ideal if and only if the zero ideal of $R^{\flat}$ is a prime ideal.\end{proof}
In the proof of Theorem \ref{Prime ideals and tilting} we will appeal to the concept of the 'spectral radical' $I_{\spc}$ of an ideal $I$ of $R$ (see Definition \ref{Spectral radical}).
\begin{proof}[Proof of Theorem \ref{Prime ideals and tilting}] It suffices to prove that the mutually inverse bijective maps $I \mapsto I^{\flat}$ and $J \mapsto J^{\#}$ send prime ideals to prime ideals. Since the map $R^{\flat} \to R, (f^{(n)})_{n \in \mathbb{N}} \mapsto f^{(0)}$, is multiplicative, ${\mathfrak{p}}^{\flat}$ is clearly a prime ideal whenever $\mathfrak{p}$ is a spectrally reduced prime ideal of $R$. To prove the converse, let $\mathfrak{q} \subsetneq{R^{\flat}}$ be a spectrally reduced prime ideal of $R^{\flat}$ and let $I, J$ be ideals of $R$ with $IJ \subseteq \mathfrak{q}^{\#}$. Since $\mathfrak{q}^{\#}$ is a radical ideal this implies $I \cap J \subseteq \mathfrak{q}^{\#}$ and since $\mathfrak{q}^{\#}$ is spectrally reduced we have $(I \cap J)_{\spc} = I_{\spc} \cap J_{\spc} \subseteq \mathfrak{q}^{\#}$ (the assignment $I \mapsto I_{\spc}$ commutes with intersections). Since for any ideal $I$ the spectral radical $I_{\spc}$ is by its definition a spectrally reduced ideal, ${I_{\spc}}^{\flat}$ is always defined. Hence ${I_{\spc}}^{\flat} \cap {J_{\spc}}^{\flat} = (I_{\spc} \cap J_{\spc})^{\flat} \subseteq{\mathfrak{q}}$. Since $\mathfrak{q}$ is prime this implies ${I_{\spc}}^{\flat} \subseteq \mathfrak{q}$ or ${J_{\spc}}^{\flat} \subseteq \mathfrak{q}$, so $I \subseteq I_{\spc} = {{I_{\spc}}^{\flat}}^{\#} \subseteq \mathfrak{q}^{\#}$ or $J \subseteq J_{\spc} = {{J_{\spc}}^{\flat}}^{\#} \subseteq \mathfrak{q}^{\#}$; consequently, $\mathfrak{q}^{\#}$ is prime. 

The closed sets for the Zariski topology on $\TopSpec(S)$ (where $S$ is one of $R$, $R^{\flat}$) are of the form $\mathcal{V}_{S}(I) = \{\, \mathfrak{p} \in \TopSpec(S) \mid \mathfrak{p} \supseteq I \,\}$, where $I$ is some spectrally reduced ideal of $S$ (if $I$ is not a spectrally reduced ideal, we can replace $I$ with its spectral radical without changing the set $\mathcal{V}_{S}(I)$). Since the maps $I \mapsto I^{\flat}$ and $J \mapsto J^{\#}$ from Proposition \ref{Spectrally reduced ideals and tilting} are inclusion-preserving, we have $\mathcal{V}_{R}(I)^{\flat} = \mathcal{V}_{R^{\flat}}(I^{\flat})$ and $\mathcal{V}_{R^{\flat}}(J)^{\#} = \mathcal{V}_{R}(J^{\#})$ for any spectrally reduced ideal $I$ of $R$ and any spectrally reduced ideal $J$ of $R^{\flat}$. This shows that the map $\TopSpec(R) \to \TopSpec(R^{\flat}), \mathfrak{p} \mapsto \mathfrak{p}^{\flat}$, is a homeomorphism.\end{proof}
\begin{rmk}As a word of caution, let us note that the map $I \mapsto I^{\flat}$ is far from being bijective for closed ideals of a perfectoid ring $R$ which are not spectrally reduced. As a simple example, consider the principal ideal $I = (X)_{R}$ of the perfectoid algebra $R = K\langle X^{1/p^{\infty}}\rangle$, for $K$ a perfectoid field. Then $I$ is closed since $X \in R$ is not a topological divisor of zero. We claim that $I^{\flat}$ is the zero ideal of $R^{\flat} = K^{\flat}\langle X^{1/p^{\infty}}\rangle$. To see this, let $f \in I^{\flat}$ and write $f = (f^{(m)})_{m}$ where $f^{(m)}$ is an element of $R$ with $f^{(m) p^{m}} = f^{\#}$. Since $f \in I^{\flat}$, we have $f^{(m)} \in I = (X)_{R}$ for all $m \geq 0$. But then $f^{\#} = f^{(m) p^{m}} \in (X^{p^{m}})_{R}$ for all $m$ whence $f^{\#} \in \bigcap_{m \geq 0}(X^{p^{m}})_{R} = (0)$. Thus we see that $f^{(m)} = 0$ for all $m$ (the ring $R$ being reduced) and, consequently, $f = 0$. \end{rmk}       
\bibliographystyle{plain} 
\bibliography{Bib} 

\textsc{Technische Universit\"{a}t M\"{u}nchen, Garching, Germany} \newline 

E-mail address: \textsf{dimas.dine@tum.de}

\end{document}